\documentclass[12pt,leqno]{amsart}
\usepackage{amsmath,amssymb,amsfonts}
\usepackage[all]{xy}

\usepackage[T5,T1]{fontenc}
\usepackage{mathpazo}
\usepackage{hyperref}
\usepackage{appendix}
\usepackage{a4wide}

\theoremstyle{plain}
\newtheorem{thm}{Theorem}[section]
\newtheorem*{mainthm}{Main Theorem}
\newtheorem{lem}[thm]{Lemma}
\newtheorem{cor}[thm]{Corollary}
\newtheorem{prop}[thm]{Proposition}

\newtheorem{condition}[thm]{Condition}

\theoremstyle{definition}

\newtheorem{rmk}[thm]{Remark}

\makeatletter
\newcommand{\neutralize}[1]{\expandafter\let\csname c@#1\endcsname\count@}
\makeatother

\newcommand{\surj}{\twoheadrightarrow}

\newcommand{\im}{{\rm im}}

\newcommand{\Gal}{{\rm Gal}}

\newcommand{\sI}{{\mathcal I}}

\newcommand{\C}{{\mathbb C}}

\newcommand{\F}{{\mathbb F}}

\newcommand{\N}{{\mathbb N}}

\newcommand{\Q}{{\mathbb Q}}

\newcommand{\U}{{\mathbb U}}

\newcommand{\Z}{{\mathbb Z}}

\newcommand{\res}{{\rm res \hspace{.1ex} }}

\def\NDT{{\fontencoding{T5}\selectfont Nguy\~ \ecircumflex n Duy T\^an}}
 
\begin{document}
\title{Relations in the maximal pro-$p$ quotients of absolute Galois groups}
\date{\today}
\dedicatory{Dedicated to John Labute}
 \author{ J\'an Min\'a\v{c}, Michael Rogelstad and \NDT}
\address{Department of Mathematics, Western University, London, Ontario, Canada N6A 5B7}
\email{minac@uwo.ca}
\address{Department of Mathematics, Western University, London, Ontario, Canada N6A 5B7}
\email{mrogelst@uwo.ca}
 \address{
 Institute of Mathematics, Vietnam Academy of Science and Technology, 18 Hoang Quoc Viet, 10307, Hanoi - Vietnam } 
\email{duytan@math.ac.vn}
\thanks{JM is partially supported  by the Natural Sciences and Engineering Research Council of Canada (NSERC) grant R0370A01. 
NDT is partially supported  by the Vietnam Academy of Science and Technology grant  {\DJ}LTE00.01/18-19.
}

 \begin{abstract}We observe that some  fundamental constructions in Galois theory can be used to obtain  interesting restrictions on the structure of Galois groups of maximal $p$-extensions of fields containing a primitive $p$-th root of unity. This is an extension of some significant ideas of Demushkin, Labute and Serre from local fields to all fields containing a primitive $p$-th root of unity. Our techniques use certain natural simple Galois extensions together with some considerations in Galois cohomology and Massey products. 
\end{abstract}
\maketitle
\section{Introduction}
The major open question in Galois theory is to describe the absolute Galois groups of fields among profinite groups. A description of the maximal pro-$p$ quotients of absolute Galois groups of general fields for a given prime number $p$ is already a challenging problem. For a field $F$, we denote by $F_{sep}$ the separable closure of $F$ in some algebraic closure of $F$. We set  $G_F=\Gal(F_{sep}/F)$,  the absolute Galois group of  $F$, and $G_F(p)$ its maximal pro-$p$ quotient. 
In the mid-1960s, some rather fascinating advances were made in the determination of $G_F(p)$ for local fields. Already in \cite{Sha47}, I. R. Shafarevich essentially showed  that $G_F(p)$ is a free pro-$p$ group if $F$ is a local field which does not contain a primitive $p$-th root of unity. (Shafarevich did not formulate this result in the language of profinite groups, as this language was introduced later.) In 1954 Y. Kawada \cite{Ka} showed that if $F$ is a local field containing a primitive $p$-th root of unity, then $G_F(p)$ admits a presentation
\[
1\to R\to S\to G_F(p)\to 1,
\]
where $S$ is a free pro-$p$ group and $R$ is a normal subgroup of $S$ generated by a single relation $r$. The challenging and extremely interesting problem of determining a possible $r$ explicitly was completely solved in a series of papers \cite{De1}, \cite{De2}, \cite{Se1} and \cite{Lab}. In fact Labute's paper \cite{Lab} completely classifies all Demushkin groups which include all $G_F(p)$, where $F$ is a local field containing a primitive $p$-th root of unity. One example of such a relation $r$ is 
\begin{equation}
\label{eq:1}
r=x_1^{p^s}[x_1,x_2]\cdots [x_{n-1},x_n],
\end{equation}
where $n$ is an even natural number and $s\in \N$.

There arises a natural question as to whether the groups $G_F(p)$ for other fields $F$
containing a primitive $p$-th root of unity can be described by relations of a similar shape.
In some previous papers including \cite{CEM}, \cite{EMi1}, \cite{EMi2},    \cite{MT2}  using the Bloch-Kato conjecture, which is now the Rost-Voevodsky theorem \cite{Voe}, or  techniques involving Massey products in Galois theory (see also \cite{Mat}, \cite{EMa}, \cite{MT1} and \cite{MT3}), it was shown that some relations which include the triple commutators $[[x_1,x_2],x_3]$ as a factor cannot be in $G_F(p)$ for a field containing a primitive $p$-th root of unity. 
The next question is about possible combinations of $p$-th powers and commutators in the shape of relations defining $G_F(p)$. 

During the summer of 2013 we obtained some ideas which showed that some simple Galois extensions obtained from $F$ by extracting suitable $p^c$-th roots of unity for different $c\in \N$ can be used to obtain interesting restrictions on the shape of products of $p$-th powers of generators and commutators in relations in $G_F(p)$. 
The idea is to produce some explicit small Galois extensions where the restrictions of the proposed relations to these Galois groups cannot possibly be valid. In retrospect these Galois extensions conceptually could be considered before the unipotent Galois extensions constructed in \cite{MTE}, \cite{MT3} or \cite{AMT}. 
The existence of these later extensions is related  to the vanishing of Massey products. (For the vanishing of triple Massey products see \cite{Mat}, \cite{EMa}, and \cite{MT1}.)
In our case the existence of our extension is governed by the structure of  roots of unity in the base field and just enough elements in the base field independent from the roots of unity. 
Again in retrospect we see that these Galois extensions are produced by extending  the techniques which were used in \cite{AS}, \cite{Be}, \cite{Wha} and others to produce some automatic large extensions, showing in particular that finite absolute Galois groups $G_F$ or finite $G_F(p)$ can only be groups of order dividing 2.


Our ideas mentioned above, form the basis of the current paper. In the thesis of M. Rogelstad \cite{Ro} Chapter 5, we described examples which well represent these ideas. In fact, as we shall see, some main theorems in our paper, including Theorem~\ref{thm:l<m} and Theorem~\ref{thm:1},  are direct extensions of the techniques presented in \cite{Ro} together with Labute's Proposition 6 in \cite{Lab}.  Let $p$ be an odd prime and $n$ an odd  positive integer. Let  $G=S/\langle r \rangle$, where $S$ is a free pro-$p$ group on generators $x_1,x_2,\ldots,x_n$, and 
\begin{equation}
\label{eq:2}
r=x_1^{p^s}[x_2,x_3]\cdots [x_{n-1},x_n],
\end{equation}
with $s\in \N$, and $\langle r\rangle $ is the smallest closed normal subgroup of $S$ which contains $r$. 
Theorem~\ref{thm:1} implies in particular that $G$ cannot be isomorphic to $G_F(p)$ for any field $F$ containing a primitive $p$-th root of unity. The case $s=1$ has already been implied by \cite[Theorem 5.1.2 and Theorem 5.2.1]{Ro}. Observe that the shape of relations (\ref{eq:1}) and (\ref{eq:2}) is quite similar. Nevertheless the difference between  these two relations is crucial. Indeed when we consider the realizability of $G=S/\langle r\rangle$ as a possible Galois group $G_F(p)$  for some field $F$ containing a primitive $p$-th root of unity, we see that we obtain different answers for the shape of $r$ in  form (\ref{eq:1}) or (\ref{eq:2}). Namely when we consider $r$ which has the shape described in (\ref{eq:1}), then the resulting group is realizable as $G_F(p)$ for some $F$ as above. 
However when we consider $r$ which has the shape described in (\ref{eq:2}), then the resulting group is not realizable as $G_F(p)$ for any such field $F$.
We  were well acquainted with \cite[Proposition 6]{Lab} and its relevance to our work. We realized that  it allows a generalization to the infinite case. (See Lemma~\ref{lem:Labute Prop 6}.) 
Throughout our paper a prominent role is played by the simple Galois extension $F(a,m)=F(\sqrt[p^m]{a},\zeta_{p^m})$ of $F$ introduced in Section 2. (See also \cite[Chapter 5]{Ro}, where we introduced these extensions for $m=1$ and 2 in our examples illustrating these ideas. The general case $m\in \N$ is an  extension of these examples.)

As we mentioned  above, in \cite{Lab}, Labute classified all Demushkin groups and in this way all $G_F(p)$, where $F$ is a local field. He provided explicit descriptions of relations in these groups.
It is  interesting to clarify to what extent we can generalize Labute's result to all fields. Our results form a contribution to this problem.
We mentioned some of these ideas to C. Quadrelli in the fall of 2013 and also in later discussions. I. Efrat and C. Quadrelli  developed a nice group-theoretical approach to this project.  
Their paper \cite{EQ} complements  our paper well, and we feel that both papers form a tribute to the remarkable thesis of John Labute.

  We hope that our paper will appeal to a  broad audience. In particular this paper should be accessible to graduate students.

The organization of our paper is as follows: In Section 2 we introduce our basic extensions $F(a,m)$ which we  substantially use throughout the paper to show that some relations in $G_F(p)$ cannot occur. In Section 3 we recall and slightly generalize parts of  Proposition 6 in \cite{Lab}. We then prove the main results,  Theorem~\ref{thm:l<m}, Theorem~\ref{thm:1}, Theorem~\ref{thm:T} and Theorem~\ref{thm:filtration},   which were previously illustrated in \cite{Ro} in a few examples.  We also summarize all of the main results in the Main Theorem at the end of this section.
In the last section, we consider a type of automatic Galois realization  (Theorem~\ref{thm:aab Galois}) and use it to also provide  some restrictions on the shape of relations in $G_F(p)$ (Theorem~\ref{thm:aab restriction}).
Finally in Appendix A we  introduce a natural union $CR(F)$ of all $F(a,m)$ for all $a\in F^\times$ and all $m\in F$, called the $p$-cyclotomic radical extension of $F$. This appendix  is an extension and continuation of Section 3.
\\
\\
{\bf Acknowledgements.} First of all we would like to thank Professor John Labute for his remarkable inspiring work, his discussions with the first-named author over a number of years often related to the structure of Galois groups, and his continuous encouragement and enthusiasm for our work. 
The first-named author would also like to thank  Professor John Tate for the extremely interesting discussions on related topics in Galois theory during a conference in 2007 in honor of Professor John Labute.
We also warmly thank other participants  of our informal seminar in the fall of 2013, namely Masoud Ataei and Claudio Quadrelli for some exciting discussions related to our project which we present in this paper. It is a great pleasure to thank Andy Schultz and Federico Pasini for allowing us to share  our previous version of this article with them, and for their valuable suggestions. It is also a great pleasure to thank Sunil Chebolu, Ido Efrat, Stefan Gille, Chris Hall and Marina Palaisti for some very interesting and stimulating discussions concerning relations in Galois groups.
We are also grateful to the referee for his/her comments and valuable suggestions which we have used to improve our exposition.
\\
\\
\noindent {\bf Notation and convention} 

We let $p$ denote a prime number and  $v_p$ the $p$-adic valuation.

If $x$ and $y$ are elements in a group,  $[x,y]=xyx^{-1}y^{-1}$ denotes the commutator of $x$ and $y$.

For a field $F$, we set $F^\times=F\setminus \{0\}$. For $a\in F^\times$, we denote by $[a]_F$, or simply $[a]$, the class of $a$ in the quotient group $F^\times/(F^\times)^p$. We let  $\sqrt[p^n]{a}$ denote a $p^n$-th root  of $a$. 

We denote by $\mu_{p^n}$ the group of $p^n$-th roots of unity, $\mu_{p^n}=\{z\in F_{sep}\mid z^{p^n}=1\}$, and set $\mu_{p^{\infty}}=\cup_{n\geq 1}\mu_{p^n}$.

Let $C_n$ denote the cyclic group of order $n$. 

For each positive integer $m$, we choose a primitive $p^m$-th root of unity $\zeta_{p^m}$ in such a way that $\zeta_{p^m}^p=\zeta_{p^{m-1}}$ for every $m=1,2,\ldots$, where $\zeta_1:=1$. 

For  a given prime number $p$, in our paper we assume that every considered base field $F$ (unless explicitly stated otherwise) satisfies the following condition.
\begin{condition}
\label{cond} If $p$ is odd then $F$ contains a primitive $p$-th root of unity $\zeta_p$. If $p=2$ then $F$ contains a primitive fourth root $\zeta_4$ of unity.
\end{condition}


\section{Some radical Galois extensions}
In this section, we assume that $\mu_{p^\infty}\not\subseteq F$. In this case let  $k$ be  the largest positive integer such that $\zeta_{p^k}\in F^\times$. Note that if $p=2$ then by Condition~\ref{cond}, $k\geq 2$.
Let $m$ be a positive integer such that $m\geq k$. Let $a\in F^\times$ such that $[a]\not\in \langle [\zeta_{p^k}]\rangle\subseteq F^\times/(F^\times)^p$, this means $a\not \in F^p\zeta_{p^k}^b$ for every $b\in \Z$.

\begin{lem} 
\label{lem:cyclic}
  We have $\Gal( F(\zeta_{p^m})/F)\simeq C_{p^{m-k}}$.
\end{lem}
\begin{proof}
 Note that $\zeta_{p^k}\not\in F^p$ and if $p=2$ then $\zeta_{2^k}\not\in -F^2$ and in particular $\zeta_{2^k}\not\in -4F^4$.  Hence the polynomial $x^{p^{m-k}}-\zeta_{p^k}$ is irreducible (by \cite[Chapter VI, Theorem 9.1]{Lan}). Therefore $[F(\zeta_{p^m}):F]=p^{m-k}$. Furthermore, one has an injection
\[
\iota\colon \Gal(F(\zeta_{p^m})/F)\hookrightarrow (\Z/p^m\Z)^\times,
\]
which sends $\sigma\in \Gal(F(\zeta_{p^m})/F)$ to $\iota(\sigma)=[n_{\sigma}]\in (\Z/p^m\Z)^\times$ with 
\[
\sigma(\zeta) =\zeta^{n_\sigma},\quad \forall \zeta\in \mu_{p^m}.
\]
If $p$ is odd, then $(\Z/p^m\Z)^\times$ is cyclic. Hence $\Gal( F(\zeta_{p^m})/F)\simeq C_{p^{m-k}}$.

If $p=2$ then from $\zeta_4=\sigma(\zeta_4)=\zeta_4^{n_\sigma}$, we see that $n_\sigma\equiv 1\pmod 4$, for all $\sigma\in \Gal(F(\zeta_{2^m})/F)$. Thus 
\[
\Gal(F(\zeta_{2^m})/F)\simeq\im(\iota)\leq \langle [5]\rangle\leq (\Z/2^m\Z)^\times= \{\pm 1\}\times \langle [5]\rangle.\]
Therefore $\Gal( F(\zeta_{2^m})/F)$ is a cyclic group of order $2^{m-k}$.
\end{proof}

\begin{lem} 
\label{lem:a}
One has $a\not\in F(\zeta_{p^m})^p$. If $p=2$ then $a\not \in -4F(\zeta_{2^m})^4$.
\end{lem}
\begin{proof}
In order to show the first statement, by Kummer theory, it is enough to show that
\[
[a] \not\in \langle [\zeta_{p^{m-1}}]\rangle\subseteq F(\zeta_{p^{m-1}})^\times/{F(\zeta_{p^{m-1}})^\times}^p.
\]
For each $l=0,1,\ldots, m-k-1$, we let $F_l=F(\zeta_{p^{k+l}})$. We prove by induction on $l$ that 
\[
[a] \text{ is not in } \langle [\zeta_{p^{k+l}}]\rangle\subseteq F_l^\times/{F_l^\times}^p. 
\]
If $l=0$, then $F_0=F$ and  $[a]\not\in \langle [\zeta_{p^k}]\rangle\subseteq F^\times/{F^\times}^p$ by our assumption on $a$. Now suppose that $l>0$ and that $[a]\not\in \langle [\zeta_{p^{k+l-1}}]\rangle\subseteq F_{l-1}^\times/{F_{l-1}^\times}^p$. We shall show that 
$[a]\not\in \langle [\zeta_{p^{k+l}}]\rangle\subseteq F_l^\times/{F_l^\times}^p$. Suppose to the contrary that 
\[ (*)\quad \quad a=\zeta_{p^{k+l}}^s f^p, \text{ for some $s\in \Z$ and some $f\in F_l^\times$}.
\] 

If $p\mid s$ then $a\in (F_l^\times)^p\cap F_{l-1}^\times$. Hence by Kummer theory, one has
\[
[a]\in \dfrac{(F_l^\times)^p\cap F_{l-1}^\times}{{F_{l-1}^\times}^p}=\langle [\zeta_{p^{k+l-1}}]\rangle \subseteq  F_{l-1}^\times/{F_{l-1}^\times}^p,
\]
a contradiction to the induction hypothesis.

Now we consider the case that $p\nmid s$. By Lemma~\ref{lem:cyclic}, one has $[F_{l-1}(\zeta_{p^{k+l}}): F_{l-1}]=p$. Hence the polynomial $h(x):=x^p-\zeta_{p^{k+l-1}}\in F_{l-1}[x]$ is irreducible and one of its roots is $\zeta_{p^{k+l}}$. Hence 
\[
N_{F_{l-1}(\zeta_{p^{k+l}})/F_{l-1}}(\zeta_{p^{k+1}}) =(-1)^p(-\zeta_{p^{k+l-1}}).
\]
 Therefore by taking norms from $F_l$ down to $F_{l-1}$ on the both sides of (*), one gets
\[
a^p= (-1)^{ps}(-\zeta_{p^{k+l-1}})^s N_{F_l/F_{l-1}}(f)^p.
\]
Thus $(-\zeta_{p^{k+l-1}})^s\in (F_{l-1}^\times)^p$. Since $p\nmid s$, this implies that $-\zeta_{p^{k+l-1}}\in (F_{l-1}^\times)^p$. 
Since $-1=(-1)^p$ if $p$ is odd and $-1=\zeta_4^2\in (F^\times)^2$ if $p=2$, we see that $\zeta_{p^{k+l-1}}\in (F_{l-1}^\times)^p$. This is a contradiction to the induction hypothesis.

Now assume further that $p=2$ and $a\in -4F(\zeta_{2^m})^4$. We write $a=-4b^4$ for some $b\in F(\zeta_{2^m})$. Then
\[
a= \zeta_4^2 2^2 b^4= (2\zeta_4b^2)^2\in F(\zeta_{2^m})^2,
\]
a contradiction. Hence $a\not\in -4F(\zeta_{2^m})^4$.
\end{proof}

For such $a$ and $m$ as above, we define $F(a,m)=F(\zeta_{p^m},\sqrt[p^m]{a})$.
Then $F(a,m)/F$ is a Galois extension as $F(a,m)$ is the splitting field of the polynomial $x^{p^m}-a$. Let $G(a,m)=\Gal(F(a,m)/F)$. 
Define  two elements $\sigma,\tau$ of $G(a,m)$ by
\[
\begin{aligned}
\tau(\zeta_{p^m})=\zeta_{p^m} &\text{ and } \tau(\sqrt[p^m]{a})=\zeta_{p^m}\sqrt[p^m]{a};\\
\sigma(\sqrt[p^m]{a})=\sqrt[p^m]{a} &\text{ and } \sigma(\zeta_{p^m})=\zeta_{p^m}^{p^k+1}.
\end{aligned}
\]
(The existence of $\sigma$ and $\tau$ will be shown in the proof of the following proposition.)
\begin{prop} 
\label{prop:presentation of G(a,m)}
The Galois group $G(a,m)$ has the following presentation
\[
G(a,m)= \langle \sigma,\tau\mid \tau^{p^m}=\sigma^{p^{m-k}}=1,\sigma\tau\sigma^{-1}=\tau^{p^k+1}\rangle\simeq C_{p^m}\rtimes C_{p^{m-k}}.
\]
\end{prop}
\begin{proof} 
By Lemma~\ref{lem:a} and by  \cite[Chapter VI, Theorem 9.1]{Lan}, $x^{p^m}-a$ is an irreducible polynomial over $F(\zeta_{p^m})$.  Hence  $F(a,m)=F(\zeta_{p^m})(\sqrt[p^m]{a})$ has degree $p^m$ over $F(\zeta_{p^m})$. 
Thus by Lemma~\ref{lem:cyclic}, we have
\[
[F(a,m):F]=[F(\zeta_{p^m})(\sqrt[p^m]{a}):F(\zeta_{p^m})][F(\zeta_{p^m}):F]=
p^mp^{m-k}=[F(\sqrt[p^m]{a}):F][F(\zeta_{p^m}):F].
\]
This implies that $F(\sqrt[p^m]{a})\cap F(\zeta_{p^m})=F$. 
By the Galois correspondence, the smallest subgroup of $\Gal(F(a,m)/F)$ containing both $\Gal(F(a,m)/F(\zeta_{p^m}))$ and $\Gal(F(a,m)/F(\sqrt[p^m]{a}))$ is the whole Galois group $\Gal(F(a,m)/F)$. Hence  
\[\Gal(F(a,m)/F(\zeta_{p^m})) \Gal(F(a,m)/F(\sqrt[p^m]{a}))= \Gal(F(a,m)/F). 
\]
Clearly, by Kummer theory, one has $\Gal(F(a,m)/F(\zeta_{p^m})\simeq C_{p^m}$. Hence there exists a generator $\tau$ in $\Gal(F(a,m)/F(\zeta_{p^m})$ such that 
\[
\tau(\sqrt[p^m]{a})=\zeta_{p^m}\sqrt[p^m]{a}.
\]
By Lemma~\ref{lem:cyclic} applied to $F(\sqrt[p^m]{a})$, we see that $\Gal(F(a,m)/F(\sqrt[p^m]{a}))$ is cyclic and hence 
 $\Gal(F(a,m)/F(\sqrt[p^m]{a}))\simeq C_{p^{m-k}}$. Thus there is  a generator $\sigma$ in $\Gal(F(a,m)/F(\sqrt[p^m]{a}))$ such that
\[
\sigma(\zeta_{p^m})=\zeta_{p^{m-k}}\zeta_{p^m}=\zeta_{p^m}^{p^k+1}.
\]
By a direct computation, we see that
\[
\sigma\tau= \tau^{p^k+1}\sigma.
\]   
Therefore 
\[
\begin{aligned}
\Gal(F(a,m)/F)&=\Gal(F(a,m)/F(\zeta_{p^m})) \rtimes \Gal(F(a,m)/F(\sqrt[p^m]{a}))\\
&=\langle \sigma,\tau\mid \tau^{p^m}=\sigma^{p^{m-k}}=1,\sigma\tau\sigma^{-1}=\tau^{p^k+1}\rangle\simeq C_{p^m}\rtimes C_{p^{m-k}}.
\qedhere
\end{aligned}
\]
\end{proof}

Recall that for a profinite group G and a prime number $p$, the descending central series $(G_i)$, the $p$-descending central series $(G^{(i)})$, and the $p$-Zassenhaus filtration $(G_{(i)})$ of $G$ are defined inductively by
\[
G_1=G, \quad G_{i+1}=[G_i,G], \quad i=2,3,\ldots,
\]
by
\[
G^{(1)} = G, \quad G^{(i+1)} =(G^{(i)})^p[G^{(i)},G], \quad i=2,3,\ldots,
\]
and by
\[
G_{(1)}=G, \quad G_{(n)}=G_{(\lceil n/p\rceil)}^p\prod_{i+j=n}[G_{(i)},G_{(j)}], \quad n=2,3\ldots,
\]
where $\lceil n/p \rceil$ is the least integer which is greater than or equal to $n/p$.  (Here for closed subgroups $H$ and $K$ of $G$, the symbol $[H, K]$ means the smallest closed subgroup of $G$ containing the commutators $[x, y] =xyx^{-1}y^{-1}, x \in H, y\in K$. Similarly, $H^p$ means the smallest closed subgroup of $G$ containing the $p$-th powers $x^p$, $x \in H$. Observe that in this notation we are omitting the traditional use of a bar to indicate closure. For example, we simply write $L$ rather than $\bar{L}$ for the closure of $L$ in $G$.)

Recall  also that a pro-$p$-group $D$ is {\it powerful} if $D/D^p$ is abelian for odd $p$ and $D/D^4$ is abelian for $p=2$. 
\begin{prop}
\label{prop:G powerful}
 Let $m\geq k$ be positive integers and $k\geq 2$ if $p=2$.
  Let $G=G(a,m)$ be the group as in Proposition~\ref{prop:presentation of G(a,m)}:
\[
G:=G(a,m)=\langle \sigma,\tau\mid \tau^{p^m}=\sigma^{p^{m-k}}=1,\sigma\tau\sigma^{-1}=\tau^{p^k+1}\rangle.
\] 
\begin{enumerate}
\item $G_{i+1}= \langle \tau^{p^{ki}}\rangle$, for all $i\geq 1$. 
\item $G$ is  powerful.
\item For each $n\geq 1$, one has $G_{(n)}=G^{p^s}$, with $p^{s-1}<n\leq p^s$.
\end{enumerate}
\end{prop}
\begin{proof}
(1) We prove by induction on $i$. For $i=1$, we have  
\[
G_2= [G,G]= \langle [\sigma,\tau]\rangle =\langle \tau^{p^k}\rangle. 
\]
Now assume that the formula is true for $i$. We have
\[
\sigma\tau^{p^{ki}}\sigma^{-1} =(\sigma\tau\sigma^{-1})^{p^{ki}}=(\tau^{p^k+1})^{p^{ki}}=\tau^{p^{k(i+1)}}\tau^{p^{ki}}.
\]
Therefore
\[
G_{i+2}=[G,G_{i+1}]=\langle[\sigma,\tau^{p^{ki}}]\rangle=\langle \tau^{p^{k(i+1)}}\rangle, 
\]
as desired.

(2) One has $G_2=[G,G] =\langle \tau^{p^k}\rangle \leq \langle \tau^p\rangle\leq G^p$ if $p$ is odd, and $G_2=[G,G]=\langle \tau^{2^k}\rangle \leq \langle \tau^4\rangle \leq G^4$ if $p=2$. Hence  $G$ is powerful.

(3) By \cite[Theorem 11.2]{DdSMS} and by (1), we have
\[
G_{(n)}=\prod_{ip^h\geq n}G_i^{p^h}= G^{p^s} \prod_{i\geq 2; ip^h\geq n} \langle\tau^{p^{k(i-1)+h}}\rangle.
\]
For $i\geq 2$ and $ip^h\geq n$, one has $p^{k(i-1)}\geq p^{i-1}\geq i$ and
\[
p^{k(i-1)+h} \geq i p^h\geq n>p^{s-1}.
\]
Hence $k(i-1)+h\geq s$. Thus $\langle\tau^{p^{k(i-1)+h}}\rangle \leq \langle\tau^{p^s}\rangle \leq G^{p^s}$ and $G_{(n)}=G^{p^s}$.
\end{proof}
\begin{prop} Let the notation be as in Proposition~\ref{prop:G powerful}.
\label{prop:exponent of G}
\begin{enumerate}
\item The exponent of $G(a,m)$ is $p^m$.
\item The smallest $n_0$ such that $G^{(n_0)}=1$ is $n_0=m+1$.
\item The smallest $m_0$ such that $G_{(m_0)}=1$ is $m_0=p^{m-1}+1$.
\end{enumerate}
\end{prop}
\begin{proof} Since $G:=G(a,m)=\langle \sigma,\tau\rangle$ is powerful, we have 
\[G^{(n)}=G^{p^{n-1}}=\langle \sigma^{p^{n-1}},\tau^{p^{n-1}}\rangle=\langle \{x^{p^{n-1}}\mid x\in G\}\rangle.
\]
From this we see that the exponent of $G$ is $m$ and  that the smallest $n_0$ such that $G^{(n_0)}=1$ is 
\[ n_0=\log_p(\text{exponent of $G$})+1=m+1.\]
\end{proof}
\section{Relations in the maximal pro-$p$ quotient of absolute Galois groups}

The following result will be used below to prove Lemma~\ref{lem:Labute Prop 6}.

\begin{lem}
\label{lem:D}
  Let $G$ be a pro-$p$-group with a minimal set of  generators $\{x_j\}_{j\in J}$. 
Then for any family $\{a_j\}_{j\in J}$ of elements in $\Z/p\Z$ having the property that $a_j\not=0$ only for a finite number  of $j\in J$, there exists a continuous homomorphism $D\colon G\to \Z/p\Z$ such that $D(x_j)=a_j$ for all $j\in J$. 
\end{lem}
\begin{proof} This follows from \cite[Theorem 6.2]{Ko}.
\end{proof}
Let $G$ be a pro-$p$-group, $\U_p=\Z_p^\times$ the group of $p$-adic units with the $p$-adic topology, and $\chi$ a continuous homomorphism of $G$ to $\U_p$. 
We define an action of $G$ on $\Z_p$ by $\sigma\cdot x=\chi(\sigma)x$ for $\sigma\in G$, $x\in \Z_p$.  Then $\Z_p$, with the $p$-adic topology, becomes a topological $G$-module which we denote by  $\sI=\sI(\chi)$. The following result is a variant of \cite[Proposition 6]{Lab}. By using the previous lemma, the proof in \cite{Lab} still works well in this case.  For  the convenience of the reader, we reproduce the proof with suitable adjustments here.

Observe that for each $i \in \N$, the module $\sI/p^i\sI$ is a discrete $G$-module. This means that for each continuous crossed homomorphism $D \colon G \to  \sI/p^i\sI$, the kernel is an open subgroup of $G$. In particular, the kernel of $D$ contains all but finitely many generators of $G$. (See \cite[Definition 4.1 and Theorem 1.22]{Ko}.) In the proof of Lemma~\ref{lem:Labute Prop 6} we use this observation.

\begin{lem}
\label{lem:Labute Prop 6}
 Consider the following two statements:
\begin{enumerate}
\item For all $m\geq 1$ the canonical homomorphism $H^1(G,\sI/p^m\sI)\to H^1(G,\sI/p\sI)$ is surjective.
\item For all $m\geq 1$ we may arbitrarily prescribe the values of crossed homomorphisms of $G$ to $\sI/p^i\sI$ on a minimal system of generators of $G$ provided we require that for all but a finite number of generators, these values are 0.
\end{enumerate}
Then $(1)$ implies $(2)$.  
\end{lem}

\begin{proof}
Observe that $G$ acts trivially on $\sI/p\sI=\Z/p\Z$ because  any continuous homomorphism from any pro-$p$-group into $(\Z/p\Z)^\times$ is trivial.
 We shall proceed our proof by induction on $i\geq 1$. If  $m=1$ then our statement follows therefore from Lemma~\ref{lem:D}.
 We shall now assume that our statement is valid for $m-1$ and prove it for $m$ using  the exact sequence
 \[
 0\to \sI/p^{m-1}\sI \stackrel{\lambda}{\to} \sI/p^m\sI \to \sI/p\sI\to 0, 
 \]
 where $\lambda$ is induced by  multiplication by $p$.
  
   Let $g_i$, $i\in I$, be a minimal system of topological generators of $G$ and let $a_i$, $i\in I$  be elements in $\sI/p^i\sI$ with $a_i=0$ for all but finitely many $i$'s.  Using (1) we can find a crossed homomorphism $D_1$ of $G$ into $\sI/p^m\sI$ such that $b_i:=D_1(g_i)-a_i\in \im(\lambda)$. 
   One has $D_1(g_i)=0$ for all but finitely many $i$'s. Thus $b_i=0$ for all but finitely many $i$'s. By the inductive hypothesis there exists a crossed homomorphism $D_2$ of $G$ into $\sI/p^{m-1}\sI$ such that
$D_2(g_i)=\lambda^{-1}(b_i)$.
 Then $D = D_1-\lambda\circ D_2$ is a crossed homomorphism of $G$ into $\sI/p^m\sI$ such that $D(g_i) = a_i$. 
\end{proof}
Now suppose that   $F$  is any field containing a primitive $p$-th root of unity. 
There exists a canonical isomorphism
\[
h\colon {\rm Aut}(\mu_{p^\infty}) \simeq \U_p,
\]
given by $\sigma(\xi)=\xi^{h(\sigma)}$. The action of $G_F(p)$ on $\mu_{p^\infty}$ is given  by a character
\[
\chi_{p,cycl}\colon G_F(p) \to \U_p.
\]
The character $\chi_{p,cycl}$ is called the $p$-cyclotomic character. For any $\sigma\in G_F(p)$, $\chi_{p,cycl}(\sigma)$ is determined by the condition that
\[
\sigma(\xi)= \xi^{\chi_{p,cycl}(\sigma)}, \quad \forall \xi\in \mu_{p^\infty}.
\]

\begin{prop}
\label{prop:surjective}
Let $\sI=\sI(\chi_{p,cycl})$.  Then for each $i\geq 1$, the canonical homomorphism 
\[
H^1(G_F(p),\sI/p^m\sI) \to H^1(G_F(p),\sI/p\sI)
\]
 is surjective.
\end{prop}
\begin{proof}
Let $F(p)$ be the compositum of all finite Galois extensions of $F$ whose degree is a power of $p$. We have $G_F(p)=\Gal(F(p)/F)$. 

Recalling that we are choosing a compatible system of the primitive $p^n$th roots of unity, we obtain an isomorphism $\mu_{p^\infty}\simeq \sI(\chi_{p,cycl})$ as a $G_F(p)$-module. From this and from the exact squence
\[
0\to \mu_{p^m}\to F(p)^\times \stackrel{p^m}{\to} F(p)^\times\to 0,
\]
we obtain a commutative diagram
\[
\xymatrix{
F^\times/{F^\times}^{p^m} \ar@{->}[r] \ar@{->}[d]& H^1(G_F(p),\mu_{p^m}) \ar@{->}[r] \ar@{->}[d] &H^1(G_F(p),\sI/p^m\sI) \ar@{->}[d]\\
F^\times/{F^\times}^{p} \ar@{->}[r]& H^1(G_F(p),\mu_{p}) \ar@{->}[r]  &H^1(G_F(p),\sI/p\sI)
}
\]
for $m\geq 1$. Since the horizontal arrows are all isomorphisms and $F^\times/{F^\times}^{p^m}\to F^\times/{F^\times}^{p}$ is surjective, we see that $H^1(G_F(p),\sI/p^m\sI) \to H^1(G_F(p),\sI/p\sI)$ is surjective.
\end{proof}

\begin{cor}
\label{cor:existence of a}
  Let  $F$ be a field containing $\zeta_p$.  Assume that $\{x\}\sqcup \{y_i\}_{i\in I}$ is a minimal system of generators for $G_F(p)$. Then for every $m\geq 1$, there exists $a\in F^\times$ and a  $p^m$-th root $\sqrt[p^m]{a}$ of $a$ such that
\[
x(\sqrt[p^m]{a})=\zeta_{p^m}\sqrt[p^m]{a} \quad \text{ and } y_i(\sqrt[p^m]{a})= \sqrt[p^m]{a} \quad \forall i\in I.
\]
\end{cor}
\begin{proof}
By   Lemma~\ref{lem:Labute Prop 6} and Proposition~\ref{prop:surjective}, there exists a crossed homomorphism $D\colon G_F(p)\to \mu_{p^m}$ such that
\[
D(x)=\zeta_{p^m} \quad\text{ and } D(y_i)=1 \quad \forall i\in I.
\]
Consider $D$ as a cocycle with values in $F(p)^\times$, then $D$ is a 1-coboundary by Hilbert's  Theorem 90. Thus there exists $\alpha\in F(p)^\times$ such that  $D(\sigma)=\sigma(\alpha)/\alpha$ for all $\sigma\in G_F(p)$. Since $\sigma(\alpha)/\alpha\in \mu_{p^m}$ for all $\sigma\in G_F(p)$, we see that $\alpha^{p^m}=:a$ is in $F^\times$. 
\end{proof}

The following theorem is a generalization of \cite[Theorem 5.1.2]{Ro} based on the same idea.

\begin{thm}
\label{thm:l<m} 
Let $F$ be a field containing $\zeta_{p^m}$ for some $m\geq 2$. Let $S$ be a free pro-$p$-group on a set of generators $X=\{x \}\cup \{y_i\mid i \in I\}$ such that
\[
1\longrightarrow R\longrightarrow S\stackrel{\pi}{\longrightarrow} G_F(p) \longrightarrow 1
\]
is a minimal presentation of $G_F(p)$.  Let $T$ be the closed subgroup of $S$ generated by $\{y_i\}_{i\in I}$.
Then there is no relation of the form $r=x^{p^lu} s\in R$, where $l$ and $u$ are integers with  $1\leq l<m$, $\gcd(p,u)=1$, and $s\in [S,S]T$.
\end{thm}
\begin{proof}
Suppose to the contrary that there is a relation $r=x^{p^lu} s$,  where $l$ and $u$ are nonzero integers with  $1\leq l<m$, $\gcd(p,u)=1$ and $s\in [S,S]T$.
  By Corollary~\ref{cor:existence of a},  we can choose $a\in F^\times$ such that 
 \[
 \pi(x)(\sqrt[p^m]{a})=\zeta_{p^m}\sqrt[p^m]{a},\quad \pi(y_i)(\sqrt[p^m]{a})=\sqrt[p^m]{a}, \forall i\in I.
\]
Since $\zeta_{p^m}\in F^\times$,  $F(\sqrt[p^m]{a})/F$ is a Galois extension with Galois group $\Gal(F(\sqrt[p^m]{a})/F)\simeq \Z/p^m\Z$.  Let ${\rm res}\colon G_F(p)\surj \Gal(F(\sqrt[p^m]{a})/F)\simeq \Z/p^m\Z$ be the restriction map. We have
\[
1= {\rm res}(\pi(r))= {\rm res}(\pi(x)^{p^lu}\pi(s))={\rm res}(\pi(x))^{p^lu},
\]
since ${\rm res}(\pi(s))=1$ for $s\in [S,S]T$. Hence the order $p^m$ of ${\rm res}\pi(x)$ divides $p^lu$. This is impossible since $m>l$.
\end{proof}

\begin{cor}
\label{cor:infinity} 
Let $F$ be a field containing $\mu_{p^\infty}$. Let $S$ be a free pro-$p$-group on a set of generators $X=\{x \}\cup \{y_i\mid i \in I\}$ such that
\[
1\longrightarrow R\longrightarrow S\stackrel{\pi}{\longrightarrow} G_F(p) \longrightarrow 1
\]
is a minimal presentation of $G_F(p)$.  Let $T$ be the closed subgroup of $S$ generated by $\{y_i\}_{i\in I}$.
Then there is no relation of the form $r=x^{p^lu} s\in R$, where $l$ and $u$ are nonzero integers with $l\geq 1$, and $s\in [S,S]T$.
\end{cor}
\begin{proof}
This follows immediately from Theorem~\ref{thm:l<m}.
\end{proof}

Let $S$ be a free pro-$p$-group on a set of generators $X=\{x \}\cup \{y_i\mid i \in I\}$ such that
\[
1\longrightarrow R\longrightarrow S\stackrel{\pi}{\longrightarrow} G_F(p) \longrightarrow 1
\]
is a minimal presentation of $G_F(p)$.

\begin{lem} 
\label{lem:act trivially} 
Let $F$ be  a field satisfying Condition~\ref{cond}. 
 Suppose that $r=x^{p^lu} s\in R$, where $s$ is in $[S,S]$ and $l$ and $u$ are nonzero integers with $l\geq 1$ and $\gcd(p,u)=1$.
  Then $\pi(x)$ acts trivially on $F(\zeta_{p^n})$ for all $n\in \N$.
 \end{lem}

\begin{proof}
If $\mu_{p^\infty}\subset F^\times$ then clearly $\pi(x)$ acts trivially on $F(\zeta_{p^n})$ for all $n\in \N$.

Now we assume that $\mu_{p^\infty}\not\subset F^\times$. Let $k$ be a positive integer such that  $\zeta_{p^k}\in F^\times$ but $\zeta_{p^{k+1}}\not\in F^\times$. 
   We proceed by induction on $n$.    
   If $n\leq k$ then $\pi(x)\in G_F(p)$ acts trivially on $\zeta_{p^n}$ since $\zeta_{p^n}=\zeta_{p^k}^{p^{k-n}}\in F$.  

Now suppose that $n>k$ and that $\pi(x)$ acts trivially on $\zeta_{p^{n-1}}$ but $\pi(x)$ acts non-trivially on $\zeta_{p^n}$.
Then the restriction of $\pi(x)$ to $F(\zeta_{p^n})$ generates the entire Galois group $\Gal(F(\zeta_{p^n})/F(\zeta_{p^{n-1}}))\simeq C_p$. Hence, by Burnside's basis theorem, the restriction of $\pi(x)$ to $F(\zeta_{p^{n+n^\prime}})$ generates the entire Galois group $\Gal(F(\zeta_{p^{n+n^\prime-1}})/F(\zeta_{p^{n-1}}))\simeq C_{p^{n^\prime}}$ for every $n=1,2,\ldots$ (Note that $F(\zeta_{p^{n-1}})$ also satisfies Condition~\ref{cond}, hence $\Gal(F(\zeta_{p^{n+n^\prime-1}})/F(\zeta_{p^{n-1}}))$ is indeed cyclic by Lemma~\ref{lem:cyclic}.)
Pick any $n^\prime >l$ and consider the restriction map $\rho\colon G_F(p)\surj \Gal(F(\zeta_{p^{n+n^\prime-1}})/F(\zeta_{p^{n-1}}))$. One has
\[
1= \rho(\pi(r))=\rho(\pi(x))^{{p^l}u} \rho(\pi(s))= \rho(\pi(x))^{{p^l}u}.
\]
Hence the order of $\rho(\pi(x))$ divides $p^l$. 
 This contradicts  the fact that the order of $\rho(\pi(x))$ is $p^{n^\prime}>p^l$.
\end{proof}

Let $S$ be a free pro-$p$ group on an alphabet $X$ of a minimal set of generators. We let $X^{-1}$ be the set of formal symbols $x^{-1}$, $x\in X$. For each  $r$ in  $[S,S]$, by {\it a commutator expression} for $r$ we mean an expression $r=c_1\cdots c_k$, where each $c_i$ is a  hyper-commutator of the form $c_i=[u_1,u_2,u_3\cdots,u_{k(i)}]$ with $u_i\in X\sqcup X^{-1}$. Here we do not specify sub-bracketing which can be arbitrary as usual when dealing with higher commutators.
We say that a commutator $[u,v]$ {\it appears} in  the commutator expression $r=c_1\cdots c_k$ for $r$ if $[u,v]$ is a sub-commutator of some hyper-commutator $c_i$. For example commutator $[u_1,u_2]$ appears in both of the elements $[[u_1,u_2],u_3]$ and $[[u_1,u_2],[u_3,u_4]]$.

The following theorem is a generalization of \cite[Theorem 5.2.1]{Ro} and our proof is based on the same idea.

\begin{thm}
\label{thm:1} 
Let $F$ be a field satisfying Condition~\ref{cond}. Let $S$ be a free pro-$p$-group on a set of generators $X=\{x \}\cup \{y_i\mid i \in I\}$ such that
\[
1\longrightarrow R\longrightarrow S\stackrel{\pi}{\longrightarrow} G_F(p) \longrightarrow 1
\]
is a minimal presentation of $G_F(p)$.
Then there is no relation of the form $r=x^{p^lu} s\in R$, where $l$  and $u$ are nonzero integers with $l\geq 1$ and $\gcd(p,u)=1$, and $s\in [S,S]$ 
 such that any commutator of the form $[u,v]$  ($u,v\in X\sqcup X^{-1}$) appearing in a fixed commutator expression for $s$ has $u\not=x^{\pm1}$ and $v\not= x^{\pm1}$.
\end{thm}
\begin{proof} 
Suppose to the contrary that there is a relation $r=x^{p^lu} s$, where $l$  and $u$ are nonzero integers with $l\geq 1$ and $\gcd(p,u)=1$, and $s\in [S,S]$.

By Corollary~\ref{cor:infinity}, we may also suppose that there exists $k\in \N$ such that $\zeta_{p^k}\in F^\times$ but $\zeta_{p^{k+1}}\not \in F^\times$.

We take any $m>\max\{k,l\}$ and choose an element $a\in F^\times$ such that
 \[
x(\sqrt[p^m]{a})=\zeta_{p^m}\sqrt[p^m]{a} \quad \text{ and } y_i(\sqrt[p^m]{a})= \sqrt[p^m]{a} \quad \forall i\in I.
\]
  Such an element $a$ exists by Corollary~\ref{cor:existence of a}. 
  By Lemma~\ref{lem:act trivially}, $\pi(x)$ acts trivially on $\zeta_{p^n}$ for every $n\in \N$. In particular, we see that  $a\not\in (F^\times)^p \zeta_{p^k}^n$ for every $n \in \Z$.
      We consider the Galois extension $F(a,m)/F$. 
 Let ${\rm res}\colon G_F(p)\surj \Gal(F(a,m)/F)$ be the restriction map.
  Clearly the order of ${\rm res}(\pi(x))\in \Gal(F(a,m)/F(\zeta_{p^m}))$ is  $p^m$. 
 By our choice of $a$, for each $i\in I$, 
 \[
 {\rm res}(\pi(y_i)) \text{ is in }\Gal(F(a,m)/F(\sqrt[p^m]{a}))\simeq C_{p^{m-k}}.
 \] 
 Since $\Gal(F(a,m)/F(\sqrt[p^m]{a}))$ is commutative, we see that ${\rm res}([\pi(y_i^{\pm 1}), \pi(y_j^{\pm1})])=1$ for all $i,j\in I$. Hence  one has ${\rm res}(s)=1$. Therefore we have
 \[
 1={\rm res}(\pi(r))= {\rm res}(\pi(x)^{p^lu}\pi(s))={\res}(\pi(x))^{p^lu},
 \]
 a contradiction to the fact that the order of ${\res}(\pi(x))$ is $p^m>p^l$.
\end{proof}

\begin{rmk} 
Let the notation be as in Theorem~\ref{thm:1}. Let $T$ be the closed subgroup of $S$ generated by $\{y_i\}_{i\in I}$. Clearly if $s\in [T,T]$ then  any commutator of the form $[u,v]$  ($u,v\in X\sqcup X^{-1}$) appearing in a fixed commutator expression for $s$ has $u\not=x^{\pm 1}$ and $v\not= x^{\pm1}$. 

For example, let  $n$ be an odd positive integer. Let $S$ be a free pro-$p$ group of generators $x_1,x_2,\ldots,x_n$ and let 
\[
r=x_1^{p^s}[x_2,x_3]\cdots [x_{n-1},x_n]
\]
with $s\in \N$. Then Theorem~\ref{thm:1} implies that $G$ is not isomorphic to $G_F(p)$ for every field $F$ satisfying Condition~\ref{cond}.
\end{rmk}

\begin{thm}
\label{thm:T}
Let $F$ be a field satisfying Condition~\ref{cond}. Let $S$ be a free pro-$p$-group on a set of generators $\{x \}\cup \{y_i\mid i \in I\}$ such that
\[
1\longrightarrow R\longrightarrow S\stackrel{\pi}{\longrightarrow} G_F(p) \longrightarrow 1
\]
is a minimal presentation of $G_F(p)$. Let $T$ be the (closed) subgroup of $S$ generated by $\{y_i\}_{i\in I}$. 
Then there is no relation of the form $r=x^{p^lu} s\in R$, where $l$ and $u$ are nonzero integers with $l\geq 1$ and $\gcd(p,u)=1$, and $s\in T$.
\end{thm}
\begin{proof} 
Suppose to the contrary that there is a relation $r=x^{p^lu} s\in R$, where $l$ and $u$ are nonzero integers with $l\geq 1$ and $\gcd(p,u)=1$, and $s\in T$.
By Corollary~\ref{cor:infinity} we  may assume that there exists $k\in \N$ such that $\zeta_{p^k}\in F^\times$ but $\zeta_{p^{k+1}}\not \in F^\times$.
We pick any positive integer $m$ with $m>l$.  By Corollary~\ref{cor:existence of a}, there exists $a\in F^\times$ such that
\[
\pi(x)(\sqrt[p^m]{a})=\zeta_{p^m}\sqrt[p^m]{a} \text{ and  }\pi(y_i)(\sqrt[p^m]{a})=\sqrt[p^m]{a}, \text{ for all $i\in I$}.\]

 We first observe that $\pi(t)(\sqrt[p^m]{a})=\sqrt[p^m]{a}$ for any $t\in T$. Then we have
\[ 
\sqrt[p^m]{a}= \pi(r)(\sqrt[p^m]{a})=\pi(x)^{p^lu}\pi(s)(\sqrt[p^m]{a})=\pi(x)^{p^lu}(\sqrt[p^m]{a}).
\]

{\bf Case 1:} {\it $\pi(x)$ acts trivially on $\zeta_{p^m}$}. Then by induction on $n$, one has $\pi(x)^n(\sqrt[p^m]{a})=\zeta^n_{p^m}\sqrt[p^m]{a}$. In particular, one has
\[
\sqrt[p^m]{a}=\pi(x)^{p^lu}(\sqrt[p^m]{a}) =\zeta_{p^m}^{p^lu} \sqrt[p^m]{a}.
\]
This implies that $\zeta_{p^m}^{p^lu}=1$ and hence $p^m$ divides $p^lu$. This is impossible because $m>l$.

{\bf Case 2:} {\it $\pi(x)$ acts nontrivially on $\zeta_{p^m}$}. One has
\[
\pi(x)(\zeta_{p^m})^{p^{m-k}}=\pi(x)(\zeta_{p^k})=\zeta_{p^k}=\zeta_{p^m}^{p^{m-k}}.
\]
Hence $\pi(x)(\zeta_{p^m})= \zeta_{p^m} \zeta_{p^{m-k}}^v=\zeta_{p^m}^{1+p^k v}$, for some $v\in \Z$. By induction on $n$, one has
\[
\pi(x)^n(\sqrt[p^m]{a})=\zeta^{1+(1+p^kv)+\cdots+(1+p^kv)^{n-1}}_{p^m}\sqrt[p^m]{a}.
\]
Hence
\[
\sqrt[p^m]{a}=\pi(x)^{p^lu}(\sqrt[p^m]{a})=\zeta^N \sqrt[p^m]{a},
\]
where
\[
\begin{aligned}
N&= 1+(1+p^kv)+\cdots+(1+p^kv)^{p^lu-1}= \dfrac{(1+p^kv)^{p^lu}-1}{p^kv}.
\end{aligned}
\]
From this, one deduces that $p^m\mid N$. On the other hand, it can be checked that for all $\alpha\in p\Z$ with $\alpha\in 4\Z$ if $p=2$, and $n\in \N$, one has \[v_p((1+\alpha)^n-1)= v_p(\alpha)+v_p(n).\]
Therefore
\[
v_p(N) = v_p(p^kv)+v_p(p^lu)-v(p^kv)=l.
\]
This  implies that  $m\leq v_p(N)=l$, a contradiction.
\end{proof}

\begin{thm}
\label{thm:filtration}
Let $F$ be a field satisfying Condition~\ref{cond}. Let $S$ be a free pro-$p$-group on a set of generators $\{x \}\cup \{y_i\mid i \in I\}$ such that
\[
1\longrightarrow R\longrightarrow S\stackrel{\pi}{\longrightarrow} G_F(p) \longrightarrow 1
\]
is a minimal presentation of $G_F(p)$. Let $T$ be the (closed) subgroup of $S$ generated by $\{y_i\}_{i\in I}$. Then there is no relation of the form $r=x^{p^{l-1}u}s t\in R$, where $l$ and $u$ are nonzero integers with  $l\geq 2$, $\gcd(p,u)=1$,  $s\in [T,T]$ and $t\in S^{(l+1)}\cap [S,S]$.
\end{thm}
\begin{proof}
Suppose to the contrary that there is a relation $r=x^{p^{l-1}u}s t\in R$, where $l$ and $u$ are nonzero integer with  $l\geq 2$, $\gcd(p,u)=1$, $s\in [T,T]$ and $t\in S^{(l+1)}\cap [S,S]$.

By Corollary~\ref{cor:infinity}, we may assume that  there exists $k\in \N$ such that $\zeta_{p^k}\in F^\times$ but $\zeta_{p^{k+1}}\not \in F^\times$.
 By Corollary~\ref{cor:existence of a}, there exists $a\in F^\times$ such that
\[
\pi(x)(\sqrt[p^l]{a})=\zeta_{p^l}\sqrt[p^l]{a} \text{ and  }\pi(y_i)(\sqrt[p^l]{a})=\sqrt[p^l]{a}, \text{ for all $i\in I$}.\]
Let ${\rm res}\colon G_F(p)\surj \Gal(F(\sqrt[p^l]{a},\zeta_{p^l})/F)$ be the restriction map.  
By Lemma~\ref{lem:act trivially}, $\pi(x)$ acts trivially on $F(\mu_{p^\infty})$.
Then ${\rm res}(\pi(x))$ has order $p^l$. 

By our choice of $a$, ${\res}(\pi(y_i))$ is in $\Gal(F(\sqrt[p^l]{a},\zeta_{p^l})/F(\sqrt[p^l]{a}))$ , which is either trivial (if $l\leq k$) or is isomorphic to $C_{p^{l-k}}$ (if $l>k$). In either case we always have
\[\res[\pi(y_i^{\pm1}),\pi(y_j^{\pm1})]=[{\res}(\pi(y_i^{\pm1})),{\res}(\pi(y_j^{\pm1}))]=1.\]
Thus ${\rm res}(s)=1$. By Proposition~\ref{prop:exponent of G}, one has ${\rm res}(\pi(t)) \in G(a,l)^{(l+1)}=\{1\}$. Therefore
\[
1={\rm res}(\pi(r))= {\rm res}(\pi(x))^{p^{l-1}u}.
\]

 This implies that the order $p^l$ of ${\rm res}(\pi(x))$  divides $p^{l-1}u$,  a contradiction. 
\end{proof}
\begin{rmk} In this previous theorem, by also using Zassenhaus filtrations we can replace the condition $t\in S^{(l+1)}$ by the (seemingly weaker) condition $t\in S^{(l+1)}\cup S_{(p^{l-1}+1)}$. However we obtain nothing new here  because, by induction on $l$, one can show that
\[
S_{(p^{l-1}+1)}\leq S^{(l+1)}.
\]
(The case $l=2$ was mentioned in \cite[page 260]{MTE}.) 
\end{rmk}

In order to summarize key results in this section we collect Theorem~\ref{thm:l<m}, Theorem~\ref{thm:1}, Theorem~\ref{thm:T} and Theorem~\ref{thm:filtration} into a single theorem as follows.
\begin{mainthm}
Let $F$ be a field satisfying Condition~\ref{cond}. Let $S$ be a free pro-$p$-group on a set of generators $\{x \}\cup \{y_i\mid i \in I\}$ such that
\[
1\longrightarrow R\longrightarrow S\stackrel{\pi}{\longrightarrow} G_F(p) \longrightarrow 1
\]
is a minimal presentation of $G_F(p)$. Let $T$ be the (closed) subgroup of $S$ generated by $\{y_i\}_{i\in I}$. Then there is no relation of the form $r=x^{p^{l}u}s \in R$, where $l$ and $u$ are nonzero integers with  $l\geq 1$, $\gcd(p,u)=1$, and
\begin{enumerate}
\item $s\in [S,S]T$ and $l<m$ if $F$ contains $\zeta_{p^m}$ for some $m\geq 2$;
\item $s\in [S,S]$ such that any commutator of the form $[u,v]$  ($u,v\in X\sqcup X^{-1}$) appearing is a fixed commutator expression for $s$ has $u\not=x^{\pm1}$ and $v\not= x^{\pm1}$;
\item $s\in T$;
\item   $s\in [T,T] (S^{(l+2)}\cap [S,S])$.
\end{enumerate}
\end{mainthm}


\section{Massey products $\langle a,\ldots,a, b\rangle$}

Let $p$ be a prime number and $k$ a positive integer less than $p$. Let $F$ be a field of characteristic $\not=p$ which contains a fixed primitive $p$-th root of unity $\zeta_p$. 
For any element $a$ in $F^\times$, we shall write $\chi_a$  for the  character corresponding to $a$ via the Kummer map $F^\times\to H^1(G_F,\Z/p\Z)={\rm Hom}(G_F,\Z/pZ)$.  Concretely $\chi_a$ is determined by
\[
\frac{\sigma(\sqrt[p]{a})}{\sqrt[p]{a}}=\zeta_p^{\chi_a(\sigma)}, \quad\forall \sigma\in G_F.
\]
The character $\chi_a$ defines a homomorphism $\chi^a\in {\rm Hom}(G_F,\frac{1}p\Z/\Z)\subseteq {\rm Hom}(G_F,\Q/\Z)$ by the formula
\[
\chi^a =\frac{1}{p} \chi_a. 
\]
Let $b$ be any element in $F^\times$.  Then the norm residue symbol can be defined to be
\[
(a,b):= (\chi^a,b):= b\cup \delta \chi^a.
\]

The cup product $\chi_a\cup \chi_b\in H^2(G_F,\Z/p\Z)$ can be interpreted as the norm residue symbol $(a,b)$. More precisely, we consider the exact sequence
\[
0\longrightarrow \Z/p\Z \longrightarrow  F_{sep}^\times \stackrel{x\mapsto x^p}{\longrightarrow} F_{sep}^\times \longrightarrow 1,
\]
where $\Z/p\Z$ has been identified with the group of $p$-th roots of unity $\mu_p$ via the choice of $\zeta_p$. As $H^1(G_F,F_{sep}^\times)=0$, we obtain
\[
0{\longrightarrow} H^2(G_F,\Z/p\Z)\stackrel{i}{\longrightarrow} H^2(G_F,F_{sep}^\times) \stackrel{\times p}{\longrightarrow} H^2(G_F,F_{sep}^\times).
\]
Then one has $i(\chi_a\cup \chi_b)=(a,b)\in H^2(G_F,F_{sep}^\times)$. (See \cite[Chapter XIV, Proposition 5]{Se2}.) 

From now on we assume that $a$ is not in $(F^\times)^p$. The extension $F(\sqrt[p]{a})/F$ is a Galois extension with Galois group $\langle \sigma_a\rangle\simeq \Z/p\Z$, where $\sigma_a$ satisfies $\sigma_a(\sqrt[p]{a})=\zeta_p\sqrt[p]{a}$. 

Assume that $a$ and $b$ are elements in $F^\times$, which are linearly independent modulo $(F^\times)^p$. Let $K= F(\sqrt[p]{a},\sqrt[p]{b})$. Then $K/F$ is a Galois extension whose Galois group is generated by $\sigma_a$ and $\sigma_b$. Here 
\[
\begin{aligned}
\sigma_a(\sqrt[p]{b})=\sqrt[p]{b}, \sigma_a(\sqrt[p]{a})=\zeta_p \sqrt[p]{a};\\ \sigma_b(\sqrt[p]{a})=\sqrt[p]{a}, \sigma_b(\sqrt[p]{b})=\zeta_p \sqrt[p]{b}. 
\end{aligned}
\]
Let $a$ and $b$ be two elements in $F^\times$ which are linearly independent modulo $p$. The extension $F_a=F(\sqrt[p]{a})$ is Galois with Galois group  generated by $\sigma_a$.


Assume that $\chi_a\cup \chi_b=0$. Then the norm residue symbol $(a,b)$ is trivial. Hence there exists $\alpha$ in $F(\sqrt[p]{a})$ such that $N_{F(\sqrt[p]{a})/F}(\alpha)=b$ (see  \cite[Chapter XIV, Proposition 4 (iii)]{Se2}). 
For each $i=0,\ldots,p-1$, we consider the following element
\[
D_i(s):=\sum_{j=0}^{p-i-1} \binom{p-j-1}{i} s^j \in \Z[s].
\]
\begin{lem}
\label{lem:Di} One has 
\[
(s-1)D_i(s)=D_{i-1}(s)-\binom{p}{i}s^0.
\]
\end{lem}
\begin{proof}
One has
\begin{align*}
(s-1)D_i(s)&= \sum_{j=0}^{p-i-1} \binom{p-j-1}{i} s^{j+1} -\sum_{j=0}^{p-i-1} \binom{p-j-1}{i} s^j\\
&=\sum_{j=1}^{p-i} \binom{p-j}{i} s^{j} -\sum_{j=0}^{p-i-1} \binom{p-j-1}{i} s^j\\
&= \sum_{j=0}^{p-i}\left( \binom{p-j}{i}- \binom{p-j-1}{i}\right) s^{j}-\binom{p}{i}s^0\\
&=\sum_{j=0}^{p-i} \binom{p-j-1}{i-1} s^j -\binom{p}{i}s^0\\
&=D_{i-1}(s)-\binom{p}{i}s^0,
\end{align*}

as desired.
\end{proof}

We define $A_i:=D_i(\sigma_a)(\alpha)\in F_a$. Clearly $A_0=D_0(\sigma_a)(\alpha)=N_{F_a/F}(\alpha)=b$.

\begin{cor}
\label{cor:sigma acts}
  One has
\[
\frac{\sigma_a(A_i)}{A_i}=\frac{A_{i-1}}{\alpha^{\binom{p}{i}}}.
\]
\end{cor}  

\begin{proof} 
This follows immediately from Lemma~\ref{lem:Di}.
\end{proof}

The following lemma is elementary. We omit the proof as it is an easy exercise.
\begin{lem}
\label{lem:linear algebra}
  Let $V$ be a vector space over a field  and $N$ a nilpotent operator on $V$. Let $k$ be a nonnegative integer and let $v$ be a vector in $V$ such that $N^kv\not=0$. Then 
\[
\{v,Nv,\cdots,N^kv\}
\]  
is linearly independent.
\end{lem}

For each integer $n\geq 3$, let $\U_n(\Z/p\Z)$ be the group of $n\times n$ upper-triangular unipotent matrices with entries in $\Z/p\Z$.
Let $E_{ij}$  be the $(k+2)\times(k+2)$-matrix such that all entries are zero except for 1's in the position $(i,j)$. We consider the following matrices in $\U_{k+2}(\Z/p\Z)$: 
\[
X=I_{k+2}+E_{1,2}+\cdots+E_{k,k+1} \quad \text{ and } Y=I_{k+2}+E_{k+1,k+2}.
\]
If $x$ and $y$ are elements in a group, we define $[x^{(i)},y]$ by induction as follows:
\[
[x^{(0)},y]=y, \;
[x^{(i)},y]= [x,[x^{(i-1)},y]], \text{ for } i\geq 1.
\]

Let $G$ be the group generated by $x,y$ subject to the relations:
\begin{enumerate}
\item $x^p=y^p=1$, $[x^{(i)},y]^p=1$ for all $i=1,\ldots,k$. 
\item $[[x^{(i)},y],y]=1$, for all $i=1,\ldots,k$ and $[x^{(k+1)},y]=1$.
\end{enumerate}
\begin{lem}
We have $|G|\leq p^{k+2}$.
\end{lem}
\begin{proof} 
By using the identity $ab=[a,b]ba$ and the relations defining $G$, 
we see that every element $g\in G$ can be written in the form
\[
g=[x^{(k)},y]^{e_k} [x^{(k-1)},y]^{e_{k-1}}\cdots [x,y]^{e_1}y^{e_0}x^{e_{-1}}, 
\]
where each $e_i\in \{0,1,\ldots,p-1\}$. Then the lemma follows.
\end{proof}

\begin{lem}
\label{lem:embedding G}
 The subgroup of $\U_{k+2}(\Z/p\Z)$ generated by $X,Y$ is isomorphic to $G$.
\end{lem}
\begin{proof} Let $H$ be the subgroup of $\U_{k+2}(\Z/p\Z)$ generated by $X$  and $Y$. By induction, one can show that
\[
[X^{(i)},Y]=I+E_{k+1-i,k+2}, \forall \;0\leq i\leq k.
\] This implies  that $|H|\geq p^{k+2}$. 
Also it is easy to check that $X,Y$ satisfy the relations defining $G$. Hence  we obtain a surjective homomorphism from $G$ to $H$ which sends $x$ to $X$ and $y$ to $Y$. Since $|G|\leq p^{k+2}\leq |H|$, we see that $G$ is isomorphic to $H$. 
\end{proof}
\begin{cor}
\label{cor:LCS}
  One has $G_{k+2}=1$ and $[x^{(k)},y]\not=1$ in $G$.
\end{cor}
\begin{proof} By the proof of Lemma~\ref{lem:embedding G}, one has an injection from $G$ into $\U_{k+2}(\Z/p\Z)_{k+2}$, which maps $x$ to $X$ and $y$ to $Y$. Since $\U_{k+2}(\Z/p\Z)_{k+2}=1$, this implies that $G_{k+2}=1$. Also since $[X^{(k)},Y]=I+E_{1,k+2}\not=1$, we see that $[x^{(k)},y]\not=1$ in $G$.
\end{proof}
The following result is a generalization of \cite[Proposition 3.3]{MT2}. For some related automatic Galois realizations see \cite{MSS} and \cite{Wat}.
\begin{thm}
\label{thm:aab Galois}
Let $a,b$ be elements in $F^\times$ which are linearly independent modulo $(F^\times)^p$ such that $\chi_a\cup\chi_b=0$.
  Let $k$ be an integer with $1\leq k\leq p-1$. The homomorphism 
\[
\bar{\rho}:=(\chi_a,\ldots,\chi_a,\chi_b)\colon G_F\to (\Z/p\Z)^k\times (\Z/p\Z)
\]
 lifts to a homomorphism $\rho\colon G_F \to \U_{k+2}(\Z/p\Z)$.
\end{thm}
\begin{proof}

Let $W^*$ be the $\F_p$-vector space in $F_a^\times/(F_a^\times)^p$ generated by $[A_i]_{F_a}$'s with $i=0,\ldots,k$. Let $L=F_a(\sqrt[p]{W^*})$. From Corollary~\ref{cor:sigma acts} we see that $W^*$ is an $\F_p[\Gal(F_a/F)]$-module. Hence $L/F$ is a Galois extension by Kummer theory.
\\
\\
\noindent {\bf Claim:} $\dim_{\F_p}(W^*)=p^{k+1}$. Hence $[L:F]=p^{k+2}$.\\
{\it Proof of claim}:  From Corollary~\ref{cor:sigma acts}, one has the relation
\[
[(\sigma_a-1)^k(A_k)]_{F_a}=[A_0]_{F_a}=[b]_{F_a}\not=0.
\] 
By Lemma~\ref{lem:linear algebra}, we see that 
\[
\{ [A_k]_{F_a},[A_{k-1}]_{F_a},\ldots,[A_1]_{F_a},[A_0]_{F_a}\}
\] 
is an $\F_p$-basis for $W^*$.
\\
\\
Since $\tilde\sigma_a(A_i)=A_i\dfrac{A_{i-1}}{\alpha^{\binom{p}{i}}}$, and $\tilde\sigma_a(b)=b$ for each extension $\tilde{\sigma}_a$ in ${\rm Gal}(L/F)$ of $\sigma_a$, we see that for each $i$, one has
 $\tilde{\sigma}_a(\sqrt[p]{A_i})=\zeta_p^{\epsilon_i}\sqrt[p]{A_i}\dfrac{\sqrt[p]{A_{i-1}}}{\alpha^{\binom{p}{i}/p}}$,
  and $\tilde{\sigma}_a(\sqrt[p]{b})=\zeta_p^{\epsilon} \sqrt[p]{b}$,  for some $\epsilon_i,\epsilon$ in $\{0,\ldots,p-1\}$. Since $[L:F]=p^{k+2}$, there is an unique extension $\tilde{\sigma}_a$ such that 
  \[
  \tilde{\sigma}_a(\sqrt[p]{A_i})=\sqrt[p]{A_i}\dfrac{\sqrt[p]{A_{i-1}}}{\alpha^{\binom{p}{i}/p}}, \forall i=1,\ldots,k;\quad
\tilde{\sigma}_a(\sqrt[p]{b})=\sqrt[p]{b}.
  \]
  Similarly, there is an extension $\tilde{\sigma}_b\in {\rm Gal}(L/F)$ of $\sigma_a$ such that
  \[
  \tilde{\sigma}_b(\sqrt[p]{A_i})=\sqrt[p]{A_i}, \forall i=1,\ldots,k;\quad \tilde{\sigma}_b(\sqrt[p]b)=\zeta_p \sqrt[p]{b}. 
  \]
  From now on, by abuse of notation, we  omit the tildes.
  We can check that $\sigma_a$ and $\sigma_b$ satisfy the relations defining $G$.
\\
\\
{\bf Claim}: $\sigma_a^p=1$.\\
{\it Proof of Claim}: Clearly $\sigma_a(\sqrt[p]{A_0})=\sigma_a(\sqrt[p]{b})=\sqrt[p]{b}=\sqrt[p]{A_0}$. 

For each $1\leq i\leq p-1$, we prove by induction on $n\geq 1$ the following formula
\[
\begin{split}
\sigma^n_a(\sqrt[p]{A_i})&=\sqrt[p]{A_{i}}^{\binom{n}{0}} \sqrt[p]{A_{i-1}}^{\binom{n}{1}}\cdots \sqrt[p]{A_{0}}^{\binom{n}{i}} 
\\
&\alpha^{-[\binom{n-1}{0}\binom{p}{i}+\cdots+\binom{n-1}{i-1}\binom{p}{1}]/p}
\sigma_a(\alpha^{-[\binom{n-2}{0}\binom{p}{i}+\cdots+\binom{n-2}{i-1}\binom{p}{1}]/p}) \cdots \sigma_a^{n-1}(\alpha^{-\binom{0}{0}\binom{p}{i}/p}),
\end{split}
\]
where $\binom{n}{i}:=0$ if $n<i$. Clearly this formula is true for $n=1$. Now suppose that $n>1$. Then by induction, one has
\[
\begin{aligned}
\sigma^n_a(\sqrt[p]{A_i})&=\sigma_a(\sigma_a^{n-1}(\sqrt[p]{A_{i}}))\\
&=\sigma_a(\sqrt[p]{A_{i}})^{\binom{n-1}{0}} \sigma_a(\sqrt[p]{A_{i-1}})^{\binom{n-1}{1}}\cdots\sigma_a(\sqrt[p]{A_{0}})^{\binom{n-1}{i}} \\
&\sigma_a(\alpha^{-[\binom{n-2}{0}\binom{p}{i}+\cdots+\binom{n-2}{i-1}\binom{p}{1}]/p})
\cdots \sigma_a^{n-1}(\alpha^{-\binom{0}{0}\binom{p}{i}/p})\\
&= \sqrt[p]{A_i}^{\binom{n-1}{0}}\sqrt[p]{A_{i-1}}^{\binom{n-1}{0}}\alpha^{-\binom{n-1}{0}\binom{p}{i}/p}
\sqrt[p]{A_{i-1}}^{\binom{n-1}{1}}\cdots \sqrt[p]{A_{0}}^{\binom{n-1}{i-1}}\alpha^{-\binom{n-1}{i-1}\binom{p}{1}/p} \sqrt[p]{A_{0}}^{\binom{n-1}{i}}\\
&\sigma_a(\alpha)^{-[\binom{n-2}{0}\binom{p}{i}+\cdots+\binom{n-2}{i-1}\binom{p}{1}]/p}\cdots\sigma_a^{n-1}(\alpha)^{-\binom{0}{0}\binom{p}{i}/p}\\
&=\sqrt[p]{A_{i}}^{\binom{n}{0}} \sqrt[p]{A_{i-1}}^{\binom{n}{1}}\cdots \sqrt[p]{A_{0}}^{\binom{n}{i}} 
\\
&\alpha^{-[\binom{n-1}{0}\binom{p}{i}+\cdots+\binom{n-1}{i-1}\binom{p}{1}]/p}
\sigma_a(\alpha^{-[\binom{n-2}{0}\binom{p}{i}+\cdots+\binom{n-2}{i-1}\binom{p}{1}]/p}) \cdots \sigma_a^{n-1}(\alpha^{-\binom{0}{0}\binom{p}{i}/p}),
\end{aligned}
\]
as desired.  Substituting 
$A_{i-\ell}=\alpha^{\binom{p-1}{i-\ell}} \sigma_a(\alpha)^{\binom{p-2}{i-\ell}}\cdots \sigma_a^{p-i+\ell-1}(\alpha)^{\binom{i-\ell}{i-\ell}}$
, one obtains   $\sigma_a^p(\sqrt[p]{A_{i}})=\sqrt[p]{A_i}$. 
\\
\\
{\bf Claim}: $\sigma_b^p=[\sigma_a^{(i)},\sigma_b]^p=[[\sigma_a^{(i)},\sigma_b],\sigma_b]=1$, for $i=1,\ldots,k$, and $[\sigma_a^{(k+1)},\sigma_b]=1$.\\
 {\it Proof of Claim}: We consider the following exact sequence
 \[
 1\to {\rm Gal}(L/F_a)\to {\rm Gal}(L/F) \to {\rm Gal}(F_a/F)\to 1.
 \]   
 By Kummer theory ${\rm Gal}(L/F_a)\simeq W^*$, which can be considered as an $\F_p$-vector space of dimension $p^{k+1}$.  

  Since ${\rm Gal}(F_a/F)$ is abelian, we see that for each $1\leq i\leq k$, $[\sigma_a^{(i)},\sigma_b]$ is in ${\rm Gal}(L/F_a)$. Clearly $\sigma_b$ is in ${\rm Gal}(L/F_a)$.  Hence $[\sigma_a^{(i)},\sigma_b]^p=[[\sigma_a^{(i)},\sigma_b],\sigma_b]=1$. By Lemma 1.3, we see that $[\sigma_a^{(k+1)},\sigma_b]=1$.

We can define a homomorphism $\psi\colon \Gal(L/F)\to \U_{k+1}(\Z/p\Z)$ by letting
\[
\sigma_a \mapsto X \text{ and } \sigma_b\mapsto Y.
\]
The composition $\rho\colon G_F\to \Gal(L/F) \stackrel{\psi}{\to}\U_{k+1}(\Z/p\Z)$ is a desired lifting of $\bar{\rho}$.
\end{proof}
We obtain immediately the following result of Sharifi. More precisely, this result is a special case of \cite[Theorem 4.3]{Sha}.
\begin{cor}[Sharifi] 
Let $a,b$ be elements in $F^\times$ such that $\chi_a\cup\chi_b=0$. Then for every integer $k$ with $1\leq k\leq p-1$, the $k+1$-fold Massey product $\langle \chi_a,\ldots,\chi_a,\chi_b\rangle$ is defined and contains 0.
\end{cor}
For our purposes, the $k+1$-fold Massey product $\langle \chi_a,\ldots,\chi_a,\chi_b\rangle$ is said to be defined and contains 0 if there is a (continuous) homomorphism $\rho\colon G_F\to \U_{k+2}(\F_p)$ such that for every $\sigma\in G_F$, one has
\[
\begin{aligned}
\rho(\sigma)_{i,i+1}&= \chi_a(\sigma), \text{ for } i=1,\ldots,k;\\
\rho(\sigma)_{k+1,k+2}&= \chi_b(\sigma).  
\end{aligned}
\]
(See \cite[Theorem 2.4]{Dwy}, and also \cite{MT2}.)


\begin{thm}
\label{thm:aab restriction}
  Let $F$ be a field containing $\zeta_p$, where $p$ is an odd prime. Let $\sigma,\tau$ be  elements in $G_F(p)$ and let  $\bar{\sigma},\bar{\tau}$ be the images of $\sigma,\tau$   in $G_F(p)/G_F(p)^p[G_F(p),G_F(p)]$. Suppose that  the embedding problem 
\[
\xymatrix{
& & &G_F(p) \ar@{->}[d]^{\varphi} \ar@{-->}[ld]\\
0\ar[r]& \Z/p\Z\ar[r] &\U_3(\Z/p\Z)\ar[r] &\Z/p\Z\times \Z/p\Z\ar[r] &0
}
\]
has a solution. Here $\varphi$ is the composition of the natural projection $$\pi\colon G_F(p)\surj G_F(p)/G_F(p)^p[G_F(p),G_F(p)]$$ with a  homomorphism $\psi\colon  G_F(p)/G_F(p)^p[G_F(p),G_F(p)] \to \Z/p\Z\times \Z/p\Z$ such that $\psi(\bar{\sigma})=(1,0)$ and $\psi(\bar{\tau})=(0,1)$.

Then  $[\sigma^{(i)},\tau]\not=1$ in $G_F(p)$, for each $i=1,2,\ldots,p-1$.
\end{thm}
\begin{proof} 
Let $\varphi_1$ be the composition of $\varphi$ with the projection on the first coordinate $\Z/p\Z\times \Z/p\Z \to \Z/p\Z$, $(u,v)\mapsto u$. Similarly let  $\varphi_2$ be the composition of $\varphi$ with the projection on the second coordinate $\Z/p\Z\times \Z/p\Z \to \Z/p\Z$, $(u,v)\mapsto v$. Clearly $\varphi_1$ and $\varphi_2$ are $\F_p$-linearly independent in ${\rm Hom}(G,\F_p)$. By the assumption on the solvability of the stated embedding problem, we see that $\varphi_1\cup \varphi_2=0$.

Let $[a]$ and $[b]$ be elements in $F^\times/(F^\times)^p$ such that
$\varphi_1=\chi_a$ and $\varphi_2=\chi_b$. Then $a$ and $b$ are $\F_p$-linearly independent in $F^\times/(F^\times)^p$ and $\chi_a\cup\chi_b=0$. Also we have
\[
\begin{aligned}
\sigma(\sqrt[p]{a})=\zeta_p \sqrt[p]{a}; \tau(\sqrt[p]{a})=\sqrt[p]{a};\\
\sigma(\sqrt[p]{b})=\sqrt[p]{b}; \tau(\sqrt[p]{b})=\zeta_p\sqrt[p]{b}.
\end{aligned}
\]

Now we consider the extension $L/F$, where $L=F_a(\sqrt[p]{W^*})$ as in the proof of Theorem~\ref{thm:aab Galois} considered in the case $k=p-1$. Then there exist $\sigma_a$ and $\sigma_b$ in $\Gal(L/F):=H$ such that $H$ is generated by $\sigma_a$ and $\sigma_b$ and the relations in $H$ are the following relations
\[
\sigma_a^p=\sigma_b^p, [\sigma_a^{(i)},\sigma_b]^p=[[\sigma_a^{(i)},\sigma_b],\sigma_b]=1, \forall i=1,\ldots,p-1, \text{ and } [\sigma_a^{(p)},\sigma_b]=1.
\]

On $F(\sqrt[p]{a},\sqrt[p]{b})$, we have that $\sigma$ (respectively, $\tau$) acts in the same way as $\sigma_a$ (respectively, $\sigma_b$). Note also that $\Gal(L/F(\sqrt[p]{a},\sqrt[p]{b}))=\Phi(H)=[H,H]=H_2$. Therefore
\[
\begin{aligned}
\sigma&\equiv \sigma_a \bmod \Phi(H), \text{ i.e.}, \sigma =\sigma_a \gamma,\\
\tau &\equiv \sigma_b \bmod \Phi(H), \text{ i.e.}, \tau =\sigma_b\delta,
\end{aligned}
\] 
for some $\gamma$ and $\delta$ in $\Phi(H)$.  
By induction on  $i$, we shall show that
\[
[\sigma^{(i)},\tau]\equiv [\sigma_a^{(i)},\sigma_b] \bmod H_{i+2}.
\]
Clearly, this statement is true for $i=0$. Now suppose that $i>0$. Then by the induction hypothesis, in $H$ we have  
\[ 
[\sigma^{(i-1)},\tau]= [\sigma_a^{(i-1)},\sigma_b] \epsilon,
\]
for some $\epsilon \in H_{i+1}$. In $H$ we have
\begin{align*}
[\sigma^{(i)},\tau] &= [\sigma, [\sigma^{(i-1)},\sigma]]= [\sigma_a\gamma,  [\sigma_a^{(i-1)},\sigma_b] \epsilon] \\
&= [\sigma_a,[\gamma,[\sigma_a^{(i-1)},\sigma_b] \epsilon]][\gamma, [\sigma_a^{(i-1)},\sigma_b] \epsilon ]  [\sigma_a,[\sigma_a^{(i-1)},\sigma_b] \epsilon]\\
&\equiv [\gamma, [\sigma_a^{(i-1)},\sigma_b] ] [\gamma,\epsilon ] [\sigma_a,[\sigma_a^{(i-1)},\sigma_b]] [\sigma_a,\epsilon]\bmod H_{i+2}\\
&\equiv  [\sigma_a,[\sigma_a^{(i-1)},\sigma_b]] =[\sigma_a^{(i)},\sigma_b] \bmod H_{i+2},
\end{align*}
as desired. 
(One has $[xy,z]=[x,[y,z]] [y,z][x,z]$ and $[x,yz]=[x,y][y,[x,z]][x,z]$.)
 Thus
\[
[\sigma^{(p-1)},\tau]
\equiv [\sigma_a^{(p-1)},\sigma_b] \bmod H_{p+1}.
\]
 By Corollary~\ref{cor:LCS}, $H_{p+1}=1$. This implies that $[\sigma^{(p-1)},\tau]=[\sigma_a^{(p-1)},\sigma_b]$. Again by Corollary~\ref{cor:LCS}, $[\sigma_a^{(p-1)},\sigma_b]\not=1$. Thus $[\sigma^{(p-1)},\tau]\not=1$ and hence $[\sigma^{(i)},\tau]\not=1$ for every $1\leq i\leq p-1$.
\end{proof}

\appendix
\section{The cyclotomic radical $p$-extensions}
Let $F$ be a field satisfying Condition~\ref{cond}. 
We set 
\[CR(F)=F(\sqrt[p^\infty]{F^\times}):=\bigcup F(\sqrt[p^m]{a},\zeta_{p^m}),
\] 
where the union is taken as $m$ runs through the set $\{1,2,\ldots\}$ and $a$ runs over the set $F^\times$. 
The field $CR(F)$ is called the ($p$-)cyclotomic radical extension of $F$. These kinds of extensions were considered also in \cite{CMQ} and \cite{Wa}.

\begin{thm}
\label{thm:CR enough roots} Let $F$ be a field containing $\mu_{p^\infty}$. Let $I$ be a set of cardinality of a basis for $F^\times/(F^\times)^p$ over $\F_p$. Then
\[
\Gal(CR(F)/F)\simeq \langle \tau_i,i\in I\mid  [\tau_i,\tau_j]=1, \forall i,j\in I\rangle\simeq \prod_{i\in I}\Z_p.
\] 
\end{thm}

\begin{proof} Let $G=\Gal(CR(F)/F)$.
We pick a basis $ [a_i], i\in I$ of  the $\F_p$-vector space $F^\times/(F^\times)^p$. Let $\tau_i, i \in I,$ be elements of $G_F(p)$ such that
\[
 \tau_i(\sqrt[p]{a_i})=\zeta_p\sqrt[p]{a_i}, \quad \tau_j(\sqrt[p]{a_i})=\sqrt[p]{a_i},\quad \forall j\not=i.
\]
Then for each $i\in I$, the restriction of $\tau_i$ to $CR(F)$, still denoted by $\tau_i$,   is in $G$. These $\tau_i$ generate $G$. 
\\
\\
{\bf Claim}: $[\tau_i,\tau_j]=1$.
 \\
{\it Proof of Claim:} It is enough to check that for every $m\in \N$, every  $l\in I$, one has $[\tau_i,\tau_j](\sqrt[p^m]{a_l})=\sqrt[p^m]{a_l}$. Clearly, for each $i\in I$, $l\in I$ and $m\in \N$, there exists $\xi_{i,l,m}\in \mu_{p^m}$ such that
\[
\tau_i(\sqrt[p^m]{a_l})= 
_{i,l,m}\sqrt[p^m]{a_l}.
\]
One has
\[
\tau_i\tau_j(\sqrt[p^m]{a_l})=\tau_i(\xi_{j,l,m}\sqrt[p^m]{a_l})=
_{j,l,m} \tau_i(\sqrt[p^m]{a_l})=\xi_{j,l,m}\xi_{i,l,m} \sqrt[p^m]{a_l}= 
\tau_j\tau_i(\sqrt[p^m]{a_l}).
\]
  {\bf Claim:} Let $\tau\in G$ and $a\in F$. If $\tau(\sqrt[p]{a})\not=\sqrt[p]{a}$ then  $\tau^n\not=1$ for all $n\in \N$.
\\
{\it Proof of Claim:} 
Write $n=p^mt$, with $m\in \N\cup\{0\}$, $t\in \Z$ and $(p,t)=1$.
Since $\tau(\sqrt[p]{a})\not=\sqrt[p]{a}$
the restriction of $\tau$ to $F(\sqrt[p]{a})$ generates the entire Galois group $\Gal(F(\sqrt[p]{a})/F)\simeq C_p$. Hence  the
restriction of $\tau$ to $F(\sqrt[p^{m+1}]{a})$ generates the entire Galois group $\Gal(F(\sqrt[p^{m+1}]{a})/F)\simeq C_{p^{m+1}}$. 
Consider the restriction map $\rho \colon \Gal(CR(F)/F)\surj \Gal(F(\sqrt[p^{m+1}]{a})/F)\simeq C_{p^{m+1}}$. Then
\[
 \rho(\tau^{n})=\rho(\tau^{p^mt})=\rho(\tau)^{p^mt}\not=1,
\]
hence $\tau^n\not=1$.
\\
\\
For each finite subset $J$ of $I$, we define  $F_J := \bigcup_{j\in J; m\in \N}F(\sqrt[p^m]{a_j})$ and $G_J=\Gal(F_J/F)$.
\\
\\
{\bf Claim}: $G_J$ is abelian and torsion free.
\\
{\it Proof of Claim:} Let $\tau$ be any nontrivial element in $G_J$. We can write
\[
\tau= (\tau_{j_1}^{\gamma_1}\cdots \tau_{j_k}^{\gamma_k})^{p^s},
\]
where $j_1,\ldots,j_k$ are in $J$, and  $\gamma_1$ is a $p$-adic unit, and $\gamma_2,\ldots, \gamma_k$ are $p$-adic integers. Set $\tilde{\tau}:=\tau_{j_1}^{\gamma_1}\cdots \tau_{j_k}^{\gamma_k}$. Then 
\[
\tilde{\tau}(\sqrt[p]{a_{j_1}})={\tau_{j_1}}^{\gamma_1}(\sqrt[p]{a_{j_1}})= {\zeta_p}^{\gamma_1}\sqrt[p]{a_{j_1}}\not=\sqrt[p]{a_{j_1}}.\]
 By the previous claim, $\tilde{\tau}$ is not a torsion element and hence $\tau$ is not a  torsion element.

From the three claims above and also observing that   $G = \varprojlim G_J$,  we see that  $G$ is a torsion free abelian pro-$p$ group. Hence by \cite[Chapter 4, Section 4.3, Theorem 4.3.4]{RZ}, one has
\[
G= \langle \tau_i,i\in I\mid  [\tau_i,\tau_j]=1, \forall i,j\in I\rangle\simeq \prod_{i\in I}\Z_p.
\qedhere
\]
\end{proof}

  \begin{thm}
  \label{thm:CR}
   Suppose that there exists $k\in \N$ such that $\zeta_{p^k}\in F^\times$ but $\zeta_{p^{k+1}}\not\in F^\times$. Then
\[
\Gal(CR(F)/F)\simeq \langle \sigma,\tau_i,i\in I\mid  [\tau_i,\tau_j]=1, \forall i,j\in I, [\sigma,\tau_i]=\tau_i^{p^k},\forall i\in I\rangle=\left(\prod_{i\in I}\Z_p\right)\rtimes \Z_p,
\]
where $\dim_{\F_p} F^\times/(F^\times)^p=\#I+1$.
\end{thm}

\begin{proof} Let $G=\Gal(CR(F)/F)$.
  We pick a basis $[\zeta_{p^k}], [a_i], i\in I$ of  the $\F_p$-vector space $F^\times/(F^\times)^p$.
For each $i\in I$, set $a_{i,1}:=a_i$, and $K_1:=F(\zeta_{p^{k+1}},\sqrt[p]{a_{i,1}}, i\in I)$. 
  
Then there exist $\sigma$, and $\tau_{i,1}\in \Gal(K_1/F)$, $i\in I$ such that
\[
\begin{aligned}
\sigma(\zeta_{p^{k+1}})&=\zeta_{p^{k+1}}^{1+p^k}, \quad \sigma(\sqrt[p]{a_i})=\sqrt[p]{a_i}, \forall i\in I,\\
\tau_{i,1}(\zeta_{p^{k+1}})&=\zeta_{p^{k+1}},\quad \tau_{i,1}(\sqrt[p]{a_{i,1}})=\zeta_p\sqrt[p]{a_{i,1}}, \quad \tau_{i,1}(\sqrt[p]{a_{j,1}})=\sqrt[p]{a_{j,1}},\forall j\not=i.
\end{aligned}
\]
Clearly 
\[
{\rm ord}(\tau_{i,1})=p, \;\forall i\in I \text{ and } [\tau_{i,1},\tau_{j,1}]=1,\;\forall i,j\in I.
\]
We pick any extension $\tilde{\sigma}\in G_F(p)$ of $\sigma$.

Since $\Gal(K_1/F)$ is of exponent $p$ and $\Gal(F(\zeta_{p^{k+2}})/F)\simeq C_{p^2}$, we see that $\zeta_{p^{k+2}}$ is not in $K_1$.  Thus $F(\zeta_{p^{k+2}})\cap K_1=F(\zeta_{p^{k+1}})$ and we have a natural isomorphism (by restriction)
\[
\Gal(K_1(\zeta_{p^{k+2}})/F(\zeta_{p^{k+1}})) \to \Gal(F(\zeta_{p^{k+2}})/F(\zeta_{p^{k+1}}))\times \Gal(K_1/F(\zeta_{p^{k+1}})).
\]
Therefore there exists  $\tau^\prime_{i,2}\in \Gal(K_1(\zeta_{p^{k+2}})/F)$ such that
\[
\tau^\prime_{i,2}|_{K_1}=\tau_{i,1} \quad \text{  and } \tau^{\prime}_{i,2}(\zeta_{p^{k+2}})=\zeta_{p^{k+2}}.
\]
For each $i\in I$, pick any extension $\tilde{\tau}_{i,2}\in G_F(p)$ of $\tau_{i,2}$.

By   Lemma~\ref{lem:Labute Prop 6}, there exists a crossed homomorphism $D_{i,2}\colon G_F(p)\to \mu_{p^2}$ such that
\[
D_{i,2}(\tilde{\sigma})=1, \; D_{i,2}(\tilde{\tau}_{j,2})=1 (\forall j\not=i) \text{ and } D_{i,2}(\tilde{\tau}_{i,2})=\zeta_{p^2}.
\]
Consider $D_{i,2}$ as a cocycle with values in $F(p)^\times$. Then $D$ is a 1-coboundary by Hilbert's  Theorem 90. Thus there exists $\alpha_{i,2}\in F(p)^\times$ such that   $D_{i,2}(g)=g(\alpha_{i,2})/\alpha_{i,2}$ for all $g\in G_F(p)$. Since $g(\alpha_{i,2})/\alpha_{i,2}\in \mu_{p^2}$ for all $g\in G_F(p)$, we see that $\alpha_{i,2}^{p^2}=:a_{i,2}$ is in $F^\times$. 
Set $\sqrt[p^2]{a_{i,2}}=\alpha$ and set $K_2:=F(\zeta_{p^{k+2}},\sqrt[p^2]{a_{i,2}},i\in I)$. We define $\tau_{i,2}:=\tilde{\tau}_{i,2}|_{K_2}\in\Gal(K_2/F)$. 
Clearly we have
\[
{\tau}_{i,2}(\zeta_{p^{k+2}})=\zeta_{p^{k+2}}, 
{\tau}_{j,2}(\sqrt[p^2]{a_{i,2}})= \sqrt[p^2]{a_{i,2}} \;(\forall i\not=j),\quad 
{\tau}_{i,2}(\sqrt[p^2]{a_{i,2}})=\zeta_{p^2} \sqrt[p^2]{a_{i,2}}.
\]
One also has
\[
{\rm ord}(\tau_{i,2})=p^2, \;\forall i\in I \text{ and } [\tau_{i,2},\tau_{j,2}]=1,\;\forall i,j\in I.
\]
Noting also that $\Gal(K_2/F)$ is of exponent $p^2$ and $\Gal(F(\zeta_{p^{k+3}})/F)\simeq C_{p^2}$, we see that $\zeta_{p^{k+3}}$ is not in $K_2$.  Thus $F(\zeta_{p^{k+3}})\cap K_2=F(\zeta_{p^{k+2}})$ and we have a natural isomorphism (by restriction)
\[
\Gal(K_2(\zeta_{p^{k+3}})/F(\zeta_{p^{k+2}})) \to \Gal(F(\zeta_{p^{k+3}})/F(\zeta_{p^{k+2}}))\times \Gal(K_2/F(\zeta_{p^{k+2}})).
\]

Inductively for each $m=1,2,\ldots,$ we can define 
$a_{i,m}\in F^\times$, $K_m=F(\zeta_{p^{k+m}},\sqrt[p^m]{a_{i,m}},i\in I)$ and $\tau_{i,m} \in\Gal(K_2/F)$ such that
\[
{\tau}_{i,m}(\zeta_{p^{k+m}})=\zeta_{p^{k+m}}, 
{\tau}_{j,m}(\sqrt[p^m]{a_{i,m}})= \sqrt[p^m]{a_{i,m}} \;(\forall i\not=j),\quad {\tau}_{i,m}(\sqrt[p^m]{a_{i,m}})=\zeta_{p^m} \sqrt[p^m]{a_{i,m}}.
\]
Clearly one has 
\[
{\rm ord}(\tau_{i,m})=p^m, \;\forall i\in I \text{ and } [\tau_{i,m},\tau_{j,m}]=1,\;\forall i,j\in I.
\]
One can check that $CR(F) =\bigcup_{m\geq 1}K_m$. 
For each $i\in I$ define $\tau_i\in \Gal(CR(F)/F)$ as follows: if $\alpha\in K_m$ then $\tau_i(\alpha):=\tau_{i,m}(\alpha)$. 
Let $H$ be the closed subgroup of $G$ generated by $\tau_i$, $i\in I$. Then $H$ is a subgroup of $\Gal(CR(F)/F(\mu_{p^\infty}))$. 
The natural map induced by restriction 
\[
H\to \Gal(F(\mu_{p^\infty})(\sqrt[p]{a_i},i\in I)/F(\mu_{p^\infty})) \simeq \prod_{i\in I} C_p,\]
is surjective. The  surjectivity and the isomorphism above follow from the explicit description of the action of $\tau_i$ on $\sqrt[p]{a_j}$. Therefore by Burnside's basis theorem (\cite[Theorem 4.10]{Ko}) and by Theorem~\ref{thm:CR enough roots}, we have
\[
H= \Gal(CR(F)/F(\mu_{p^\infty}))=\langle \tau_i,i\in I\mid  [\tau_i,\tau_j]=1, \forall i,j\in I\rangle\simeq \prod_{i\in I}\Z_p.
\]

Let $\varphi=\chi_{p,cylc}\colon G_F(p)\to \U_p$ be the $p$-cyclotomic character of $F$.  Pick any $\tau\in H$. 
For any $a\in F^\times$, any $m\in \N$ and any $p^m$-root $\sqrt[p^m]{a}$  of $a$, we can write 
\[
\tilde\sigma(\sqrt[p^m]{a})=\xi \sqrt[p^m]{a},\quad
\tau(\sqrt[p^m]{a})=\eta \sqrt[p^m]{a},
\]
for some $\xi, \eta \in \mu_{p^m}$. Then one has
\[
\tilde\sigma\tau(\sqrt[p^m]{a})=\tilde\sigma(\eta\sqrt[p^m]{a})=\tilde\sigma(\eta) \tilde\sigma(\sqrt[p^m]{a})= \eta^{\varphi(\tilde\sigma)}\xi\sqrt[p^m]{a},
\]
and
\[
\tau^{\varphi(\tilde\sigma)}\tilde\sigma(\sqrt[p^m]{a})=\tau^{\varphi}(\xi\sqrt[p^m]{a})=\tau^{\varphi(\tilde\sigma)}(\xi) \tau^{\varphi(\tilde\sigma)}(\sqrt[p^m]{a})=\xi \eta^{\varphi(\tilde\sigma)}\sqrt[p^m]{a}.
\]
Therefore $\tilde\sigma\tau=\tau^{\varphi(\tilde\sigma)}\tilde\sigma$ and $G=H\rtimes \langle \tilde \sigma\rangle\simeq \left(\prod_{i\in I}\Z_p\right)\rtimes \Z_p$.

Now we write $\varphi(\tilde\sigma)=1+p^ku$ with $u\in \Z_p^\times$. 
Let $\log$ and $\exp$ denote the $p$-adic logarithm function and the $p$-adic exponential function respectively. (See \cite[Chapter 5, Section 5]{Neu}.)  For each $n\geq 1$, let $U^{(n)}=1+p^n\Z_p$ the $n$-th higher unit group.
Then by \cite[Proposition 5.5]{Neu}, for $n>\dfrac{1}{p-1}$, the two functions $\exp$ and $\log$ yield two mutually inverse isomorphisms
\begin{equation*}
\xymatrix{
p^n\Z_p \ar@<+2pt>[r]^{\log} & U^{(n)} \ar@<+2pt>[l]^{\exp}.
}
\end{equation*}
In our case, one has $k>\dfrac{1}{p-1}$ by Condition~\ref{cond}. Therefore $\log(1+p^k)$ and $\log(1+p^ku)$ are both in $p^k\Z_p\setminus p^{k+1}\Z_p$. Set $v=\log(1+p^k)/\log(1+p^ku)$ then $v\in \Z_p$ and  $1+p^k=(1+p^ku)^v$. Set $\sigma:=\tilde\sigma^v$. Then
\[
\sigma \tau\sigma^{-1}=\tilde\sigma^v\tau\tilde\sigma^{-v}=\tau^{(1+p^ku)^v}=\tau^{1+p^k}.
\]
 Thus
\[
G\simeq \langle \sigma,\tau_i,i\in I\mid  [\tau_i,\tau_j]=1, \forall i,j\in I, [\sigma,\tau_i]=\tau_i^{p^k},\forall i\in I\rangle=\left(\prod_{i\in I}\Z_p\right)\rtimes \Z_p.
\qedhere
\]
\end{proof}

Let $\varphi=\chi_{p,cylc}\colon G_F(p)\to \U_p$ be the $p$-cyclotomic character of $F$.
  
\begin{cor}
Let the notation be as in the previous theorem. Then $G_{CR(F)}(p)$ is the closed subgroup of $G_F(p)$ generated by $[\sigma,\tau]\tau^{1-\varphi(\sigma)}$ with $\tau\in \ker\varphi$ and $\sigma\in G_F(p)$.
\end{cor}

\begin{proof}
 Let $H$ be the closed subgroup of $G_F(p)$ generated by $[\sigma,\tau]\tau^{1-\varphi(\sigma)}$ with $\tau\in \ker\varphi$ and $\sigma\in G_F(p)$. 
Note that for every $\gamma\in G_F(p)$, we have
$
\varphi(\gamma \sigma \gamma^{-1})= \varphi(\sigma).
$
Hence 
\[
\gamma [\sigma,\tau]\tau^{1-\varphi(\sigma)} \gamma^{-1}= [\gamma\sigma\gamma^{-1},\gamma\tau\gamma^{-1}] (\gamma\tau\gamma^{-1})^{1-\varphi(\gamma\sigma\gamma^{-1})}.
\]
Therefore we see that $H$ is  a normal subgroup of $G_F(p)$.

Set $L=CR(F)$. 
  We shall first show that the restriction map ${\rm res}\colon G_F(p)\surj \Gal(L/F)$ takes $H$ to $1$, this means $H\leq \ker (\res)$. To show this it is enough to show that ${\rm res}( [\sigma,\tau]\tau^{1-\varphi(\sigma)})$, where $\tau\in \ker\varphi$ and $\sigma\in G_F(p)$, is the identity on each field extension $F(\sqrt[p^m]{a},\zeta_{p^m})$ of $F$. 
  By abuse of notation, we also use ${\rm res}$ to denote the restriction map $\res\colon G_F(p)\surj \Gal(F(\sqrt[p^m]{a},\zeta_{p^m})/F)$. 
  
  If $m\leq k$ then $F(\sqrt[p^m]{a},\zeta_{p^m})/F$ is a cyclic Galois extension of degree $p^m$ whose Galois group is generated by $\tau_a$ defined by
$
  \tau_a(\sqrt[p^m]{a})=\zeta_{p^m}\sqrt[p^m]{a}.
  $
Then $\res(\tau)=\tau_a^\lambda$, and $\res(\sigma)=\tau_a^\mu$, for some $\lambda,\mu\in \N$. From $\zeta_{p^m}=\sigma(\zeta_{p^m})=\zeta_{p^m}^{\varphi(\sigma)}$, we see that
$
\varphi(\sigma)=1+ p^ml, \text{ for some $l\in \Z_p$}.
$
Hence $\tau_a^{1-\varphi(\sigma)}=\tau_a^{-p^ml}=1$ and
\[
{\rm res}( [\sigma,\tau]\tau^{1-\varphi(\sigma)})=[\res\sigma,\res\tau](\res\tau)^{1-\varphi(\sigma)}=\tau_a^{(1-\varphi(\sigma))\lambda}=1.
\]

Now we assume that $m>k$. We denote $F(a,m)=F(\sqrt[p^m]{a},\zeta_{p^m})$ as in Section 2. 
Let us write $\tau_a$ and $\sigma_a$, instead of $\tau$ and $\sigma$, as the original generators of $G(a,m)$ in Proposition~\ref{prop:presentation of G(a,m)}. Since $\tau\in \ker(\varphi)$, $\tau(\zeta_{p^m})=\zeta^{p^m}$. Hence $\res\tau$ is in $\Gal(F(a,m)/F(\zeta_{p^m}))$ and $\res\tau =\tau_a^\lambda$  for some $\lambda\in \N$. Also $\res\sigma= \sigma_a^\mu\tau_a^\nu$ for some $\mu,\nu\in \N$. 
One has
\[
\zeta_{p^m}^{\varphi(\sigma)}=\sigma(\zeta_{p^m})=(\res\sigma)(\zeta_{p^m})=\sigma_a^\mu(\zeta_{p^m})=\zeta_{p^m}^{(1+p^k)^\mu}.
\]
Hence $\varphi(\sigma)= (1+p^k)^\mu+p^ml$, for some $l\in \Z_p$. Thus
\[
(\res\tau)^{\varphi(\sigma)}=\tau_a^{\lambda[(1+p^k)^\mu+p^ml]}=\tau_a^{\lambda(1+p^k)^\mu}.
\]
On the other hand, from $\sigma_a\tau_a\sigma_a^{-1}=\tau_a^{p^k}$ and by induction on $\mu$ and $\lambda$ we see that
\[
\sigma_a^{\mu}\tau_a^\lambda \sigma_a^{-\mu}=\tau_a^{\lambda(1+p^k)^\mu}.
\]
Hence
\[
[{\rm res}(\sigma),{\rm res}(\tau)]=[\sigma_a^{\mu}\tau_a^{\nu},\tau_a^\lambda]=[\sigma_a^\mu,\tau_a^\lambda]=
\tau_a^{\lambda((1+p^k)^\mu-1)}= {\rm res}(\tau)^{\varphi(\sigma)-1}.
\]
Therefore
\[
{\rm res}( [\sigma,\tau]\tau^{1-\varphi(\sigma)})=[\res\sigma,\res\tau](\res\tau)^{1-\varphi(\sigma)})=1.
\]

 Now let $S$  be the free pro-$p$ group with generators $x,y_i, i\in I$.   Let $\sigma$ and $\rho$ be the elements in $G_F(p)$ defined as in the proof of Theorem~\ref{thm:CR}. In particular $\varphi(\sigma)=1+p^k$.
Let $\pi\colon S\surj G_F(p)$ be the homomorphism such that $\pi(x)=\sigma$ and $\pi(y_i)=\tau_i$ for every $i\in I$. 
 Let $\rho\colon S\to G$ be the composition map $\pi\circ{\res}$.
  Let $\tilde H:= \pi^{-1}(H)$. Then $\tilde{H}\leq \ker\rho$ because $H\leq \ker\res$.
 
Let $H^\prime$ be  the normal closed subgroup of $S$ generated (as a normal subgroup) by $[\tau_i,\tau_j]$ and $[\sigma,\tau_i]\tau_i^{-p^k}$. Clearly $H^\prime\leq \tilde{H}$. On the other hand, by Theorem~\ref{thm:CR}, $\ker \rho=H^\prime$. Thus one has
\[ H^\prime\leq \tilde{H}\leq \ker\rho=H^\prime.\] Therefore $\tilde{H}=\ker\rho$ and hence $H=\ker({\rm res})=G_{CR(F)}(p)$, as desired.
\end{proof}

\begin{prop} 
\label{prop:not in Frattini}
 Let $F$ be a field containing $\zeta_{p^k}$ for some $k\geq 1$. 
Let $\sigma$  be an element in $G_F(p)^{ab}\setminus pG_F(p)^{ab}$. Then if  ${p^s}\sigma=0$  we have $p^s\geq p^k$.  
\end{prop}
\begin{proof}
Because $\sigma$ is not in $pG_F^{ab}(p)$,  the Frattini subgroup of $G_F^{ab}(p)$, we see that there exists $a\in F^\times$ such that $\sigma(\sqrt[p]{a})\not=\sqrt[p]{a}$. 
Since $\mu_{p^k}\subseteq F^\times$, we see that $F(\sqrt[p^k]{a})/F$ is a cyclic extension by Kummer theory. By \cite[Chapter VI, Theorem 9.1]{Lan}, we see that $[F(\sqrt[p^k]{a}):F]=p^k$. 
Let ${\rm res}\colon G_F^{ab}(p)\to \Gal(F(\sqrt[p^k]{a})/F)\simeq C_{p^k}$ be the map induced by restriction. Since the restriction of $\sigma$ to $F(\sqrt[p]{a})$ generates the whole Galois group $\Gal(F(\sqrt[p]{a})/F)\simeq C_p$, we see that ${\rm res }(\sigma)$ generates  $\Gal(F(\sqrt[p^k]{a})/F)\simeq C_{p^k}$. This implies that $s\geq k$. 
\end{proof}

\begin{cor}
\label{cor:A5} Let $F$ be a field containing $\zeta_{p^k}$ for some $k\geq 1$. Suppose that $G_F(p)$ is a finitely generated pro-$p$-group. Then 
\[
G_F(p)^{ab} \simeq (\Z_p)^r \times \prod_{i=1}^l \Z/p^{s_i}\Z,
\]
where $r$ and $l$ are in $\N\cup \{0\}$ and $s_i\geq k$. 
Moreover if $\zeta_{p^k}$ is in $F^\times$ for every $k\geq 1$ then $l=0$.
\end{cor}
\begin{proof} 
Since $G_F(p)$ is finitely generated,  
 $G_F(p)^{ab}$ is also finitely generated and by \cite[Theorem 4.2.4]{RZ}, we have
\[
G_F(p)^{ab} \simeq (\Z_p)^r \times \prod_{i=1}^l \Z/p^{s_i}\Z,
\]
where $r,l\geq 0$ and $s_i\geq 1$. 
Then Proposition~\ref{prop:not in Frattini} implies that $s_i\geq k$ for all $i=1,\ldots, l$. If $\zeta_{p^k}$ is in $F^\times$ for every $k\geq 1$ then $l=0$.
\end{proof}

\begin{rmk} Here we provide a few examples of field satisfying the hypothesis of Corollary~\ref{cor:A5}. We will fix a natural number $k\geq 1$.
\begin{enumerate}
\item[(a)] Any local field $F$ containing $\Q_p(\zeta_{p^k})$ satisfies the hypothesis. Indeed, from \cite[Chapter VII, Theorem 7.5.11]{NSW} we see that $G_F(p)$ is finitely generated.
\item[(b)] Let $F= \C((t_1))\cdots ((t_n))$, $n\in \N$, be the field of iterated formal power series over complex numbers. Then $\zeta_{p^l}$ is in $F$ for all $l\in \N$. Further one can show that $F^\times/(F^\times)^p$ has $p^n$ elements. Thus $F$ also satisfies the hypothesis of Corollary~\ref{cor:A5} and $F$ contains $\zeta_{p^l}$ for all $l\in \N$.

We can modify this example to consider $F=\F_q((t_1))\cdots((t_n))$, $n\in \N$, where $q$ is a prime power such that $p^k\mid q-1$. Then $F^\times/(F^\times)^p$ has $p^{n+1}$  elements and $\zeta_{p^k}$ is in $F$.
 
\item[(c)] Finally another family of desired fields $F$ can be obtained by considering first a field $K$ containing $\zeta_{p^k}$. Then choose a finitely generated subgroup $H$ of $G_F(p)$ and set $F$ to be the subfied of $K(p)$ fixed by $H$. 
\end{enumerate}
\end{rmk}

\begin{cor}
 Let $F$ be a field containing $\zeta_{p^k}$ for some $k\geq 1$. Then
 \[
 G_F^{[k]}:= G_F(p)^{ab}/p^kG_F(p)^{ab} \simeq \prod_{I_k} C_{p^k}
 \]
 for some index set $I_k$.
 \end{cor}
 \begin{proof} Let $L= F(\sqrt[p^k]{a}\mid a\in F^\times)$. By Kummer theory, $\Gal(L/F)=G_F^{[k]}$ is the Pontrjagin dual 
 $ (F^\times/(F^\times)^{p^k})^*$  of
 \[
F^\times/(F^\times)^{p^k}\simeq \left(\oplus_{I_1} C_p\right) \oplus \left(\oplus_{I_2} C_{p^2}\right) \oplus \cdots \oplus \left(\oplus_{I_k} C_{p^k}\right)
 \]
 for some index sets $I_1,I_2,\ldots,I_k$ \cite[Theorem 6, p.17]{Kap}. Therefore by \cite[Lemma 2.9.4 and Theorem 2.9.6]{RZ} we have
 \[
 G_F^{[k]}\simeq  (F^\times/(F^\times)^{p^k})^*\simeq 
 \left(\prod_{I_1} C_p\right) \times \left(\prod_{I_2} C_{p^2}\right) \times \cdots \times \left(\prod_{I_k} C_{p^k}\right).
 \]
 Because $\mu_{p^k}\subseteq F^\times$, we see that each cyclic extension $E/F$ of degree $p^i$ with $1\leq i<k$ embeds into a cyclic extension $K/F$ of degree $p^k$. Therefore $I_j=\emptyset$ for all $j=1,\ldots,k-1$ and $G_F^{[k]}\simeq \prod_{I_k} C_{p^k}$.
  \end{proof}
 
\begin{thm} 
 Let $F$ be a field containing $\zeta_{p^k}$ for some $k\geq 1$. Assume that 
 \[
 {\rm Tor}(G_F(p)^{ab})=\{\sigma\in G_F(p)^{ab}\mid \sigma^{k(\sigma)}=1 \text{ for some }k(\sigma)\in \N\}
 \]
 is a closed subgroup of $G_F(p)^{ab}$. Then there exist a set $J$, an integer $l\in \N$ and cardinal numbers $m(i)$, $k\leq i\leq l$ such that
 \[
 G_F(p)^{ab}=\prod_{J}\Z_p \times \prod_{i=1}^l  \prod_{m(i)}\Z/p^{s_i}\Z,
 \]
 where $k\leq s_1<s_2<\cdots<s_l$.
\end{thm}
\begin{proof} By our assumption, we see that $ {\rm Tor}(G_F(p)^{ab})$ is an abelian torsion pro-$p$ group. Thus by \cite[Lemma 4.3.7]{RZ}, there exists $s\in \N$ such that $\sigma^{p^s}=1$ for all $\sigma\in  {\rm Tor}(G_F(p)^{ab})$.
Therefore by Proposition~\ref{prop:not in Frattini}, the exponents of elements in  ${\rm Tor}(G_F(p)^{ab})$ have the form $p^{s_i}$, where
\[
p^k\leq p^{s_1}<p^{s_2}<\cdots<p^{s_l}=p^s.
\]
Let $S=  G_F(p)^{ab}/{\rm Tor}(G_F(p)^{ab})$. Then $S$ is a torsion-free abelian pro-$p$ group. Hence by \cite[Theorem 4.3.4]{RZ}, $S$ is a free abelian pro-$p$-group. Thus 
$G_F(p)^{ab}=S\oplus  {\rm Tor}(G_F(p)^{ab})$.  By considering Pontrjagin's dual of $ {\rm Tor}(G_F(p)^{ab})$ and using the fact that any discrete abelian group of bounded order is a direct sum of cyclic groups (\cite[Theorem 6, page 17]{Kap}) as well as Pontrjagin's duality theorem, we conclude that
\[
 {\rm Tor}(G_F(p)^{ab})=\prod_{k\leq i\leq l}\left( \prod_{m(i)} \Z/p^{s_i}\Z\right).
\]
Finally using the fact that $S\simeq \prod_{J}\Z_p$ for some index set $J$ (\cite[Theorem 4.3.3]{RZ}) we conclude that
\[
G_F(p)^{ab}=\prod_{J}\Z_p \times \prod_{i=1}^l \prod_{m(i)}\Z/p^{s_i}\Z,
\]
as required.
\end{proof}

\end{document}